\documentclass[]{article}
\usepackage[shortlabels]{enumitem}
\usepackage{todonotes}
\usepackage{graphicx}
\usepackage{amsmath, amssymb, amsthm}
\usepackage{bbm}
\usepackage{fullpage}

\usetikzlibrary{calc,shapes, backgrounds}
\tikzset{every node/.style={draw, fill=white, circle, minimum size = 4pt, inner sep = 0pt}}
\newcommand{\resizetikz}[2]{
	{\tikzset{every picture/.style={scale=#1}}
		#2}}

\DeclareMathOperator{\diam}{diam}

\def\Int{\textrm{Int}}
\def\Ext{\textrm{Ext}}

\usepackage[style=numeric, backend=bibtex8]{biblatex}
\bibliography{refs}
\AtEveryBibitem{\clearfield{issn}}
\AtEveryBibitem{\clearfield{doi}}
\AtEveryBibitem{\clearfield{url}}

\allowdisplaybreaks

\newtheorem{lemma}{Lemma}[section]
\newtheorem{theorem}{Theorem}[section]
\newtheorem{corollary}{Corollary}[section]
\newtheorem{definition}{Definition}[section]

\newtheorem{conjecture}{Conjecture}[section]

\newtheorem{observation}{Observation}

\title{On the classification of regular graphs with positive Lin-Lu-Yau curvature}
\author{George Brooks \thanks{University of South Carolina, Columbia, SC, USA. ({\tt ghbrooks@email.sc.edu}). The author is partially supported by NSF DMS 2038080 grant.}
\and 
Nathanael Johnson \thanks{The Ohio State University, Columbus, OH, USA. ({\tt nathan.erikson.9701@gmail.com}). The author is partially supported by NSF DMS 2038080 grant through a summer REU program.}
\and
William Linz \thanks{University of Memphis, Memphis, TN, USA. ({\tt wlinz1@memphis.edu}). The author was partially supported by NSF DMS 2038080 grant while at University of South Carolina.}
\and 
Xiaonan Liu \thanks{Louisiana State University, LA, USA. ({\tt xliu20@lsu.edu}).}
\and
Linyuan Lu \thanks{University of South Carolina, Columbia, SC, USA. ({\tt lu@math.sc.edu}). The author is partially supported by NSF DMS 2038080 grant.}
\and
Wynton Vaughn \thanks{Clemson University, Clemson, SC, USA. ({\tt wyntonv@g.clemson.edu}). The author is partially supported by NSF DMS 2038080 grant through a summer REU program.}}
\date{\today}

\begin{document}

\maketitle

\begin{abstract}
We prove several new structural and classification results about $d$-regular graphs with positive Lin--Lu--Yau (LLY) curvature. We show that any positively curved $d$-regular graph has diameter at most $2d-2$, which improves the previously best known diameter bound obtained from the Bonnet-Myers-type theorem for positively curved graphs. We further show that every positively curved $d$-regular graph is $3$-connected. We classify all positively curved $3$-regular graphs, as well as all positively curved regular planar graphs.  
\end{abstract}

\section{Introduction}

Ricci curvature has long been an important tool in the study of Riemannian manifolds. Many different discrete analogues of the Ricci curvature for graphs and metric spaces have been proposed and studied; see \textit{e.g.} \cite{BE1985}, \cite{F2003}, \cite{O2009}. Lin, Lu and Yau~\cite{LLY2011} defined a notion of discrete Ricci curvature (now called LLY curvature) on the vertex pairs of a graph $G$ (see Section 2.3 for precise definitions). Throughout this paper, a graph $G$ is called \textit{positively curved} if it has positive LLY curvature on every vertex pair of $G$. Unless specified otherwise, we assume that the graph $G$ is simple and connected.
 
 In this paper, we prove several new results about the connectivity and diameter of $d$-regular graphs with positive LLY curvature. We also completely classify all $3$-regular graphs with positive LLY curvature, as well as all positively curved regular planar graphs.  

\subsection{Positively curved $3$-regular graphs}

There has been much interest in classifying various families of graphs satisfying some condition on LLY curvature. For example, Cushing \textit{et al.} \cite{CKLLLY2021, CKLLLY2021err} classified all $3$-regular Ricci-flat graphs with girth at least $5$ (a graph is \emph{Ricci-flat} if it has LLY curvature $0$ on every edge). We can classify all $3$-regular graphs with positive LLY curvature. 

\begin{theorem}\label{alld3}
There are six positively curved $3$-regular graphs: $K_4$,  $C_3\square P_2$,  $C_4\square P_2$, $C_5\square P_2$, the Wagner graph and $K_{3, 3}$.
\end{theorem}

\begin{figure}[h]
\begin{center}
    \resizetikz{0.6}{\begin{tikzpicture}[scale=1]
  
  \foreach \x/\num in {a/0,b/1,c/2} {
    \node (\num) at ({120*\num+210}:1.5) {};  }
    \node (3) at (0:0) {};
    \foreach \i in {0,1,2,3}{
        \foreach \j in {0,1,2,3}{
            \ifnum \i < \j
            \draw (\i) -- (\j);
            \fi
        }
    }
\end{tikzpicture}}
    \resizetikz{0.6}{\begin{tikzpicture}[scale=1]

\node (0) at (210:0.5) {};
\node (1) at (330:0.5) {};
\node (2) at (90:0.5) {};
\node (3) at (210:1.5) {};
\node (4) at (330:1.5) {};
\node (5) at (90:1.5) {};
\foreach \i in {0,1,2}{
\foreach \j in {0,1,2}{
\ifnum \i < \j
\draw (\i) -- (\j);
\fi
}
}
\foreach \i in {3,4,5}{
\foreach \j in {3,4,5}{
\ifnum \i < \j
\draw (\i) -- (\j);
\fi
}
}
\draw (0) -- (3);
\draw (1) -- (4);
\draw (2) -- (5);
\end{tikzpicture}}
    \resizetikz{0.75}{\begin{tikzpicture}[scale=1]
\foreach \k in {0,...,7}{
    \ifnum \k<4
        \node (\k) at ({45+90*\k}:0.5) {};
    \else
        \node (\k) at ({45+90*\k}:1.25) {};
    \fi
}
\draw (0) -- (1);
\draw (1) -- (2);
\draw (2) -- (3);
\draw (3) -- (0);
\draw (4) -- (5);
\draw (5) -- (6);
\draw (6) -- (7);
\draw (7) -- (4);
\draw (0) -- (4);
\draw (1) -- (5);
\draw (2) -- (6);
\draw (3) -- (7);
\end{tikzpicture}}
    \resizetikz{0.75}{\begin{tikzpicture}[scale=1]
\foreach \k in {0,...,9}{
    \ifnum \k<5
        \node (\k) at ({90+72*\k}:0.4) {};
    \else
        \node (\k) at ({90+72*\k}:1) {};
    \fi
}
\draw (0) -- (1);
\draw (1) -- (2);
\draw (2) -- (3);
\draw (3) -- (4);
\draw (4) -- (0);
\draw (5) -- (6);
\draw (6) -- (7);
\draw (7) -- (8);
\draw (8) -- (9);
\draw (9) -- (5);
\draw (0) -- (5);
\draw (1) -- (6);
\draw (2) -- (7);
\draw (3) -- (8);
\draw (4) -- (9);
\end{tikzpicture}}
    \resizetikz{0.75}{\begin{tikzpicture}[scale = 1]
\foreach \k in {0,...,7}{
    \ifnum \k<4
        \node (\k) at ({45+90*\k}:0.5) {};
    \else
        \node (\k) at ({45+90*\k}:1.25) {};
    \fi
}
\draw (0) -- (1);
\draw (1) -- (2);
\draw (2) -- (3);
\draw (3) -- (0);
\draw (4) -- (5);
\draw (5) -- (6);
\draw (6) -- (7);
\draw (7) -- (4);
\draw (0) -- (5);
\draw (1) -- (4);
\draw (2) -- (6);
\draw (3) -- (7);
\end{tikzpicture}}
    \resizetikz{0.65}{\begin{tikzpicture}[scale=1]
\foreach \k in {0,1,2}{
    \node (v\k) at (0, \k) {};
}
\foreach \k in {3,4,5}{
    \node (v\k) at (2, {\k - 3}) {};
}
\foreach \k in {0,1,2}{
    \foreach \l in {3,4,5}{
    \draw (v\k) -- (v\l);
    }
}
\end{tikzpicture}}
    
\end{center}
\caption{The six positively curved $3$-regular graphs.}
\label{fig:positive3regular}
\end{figure}
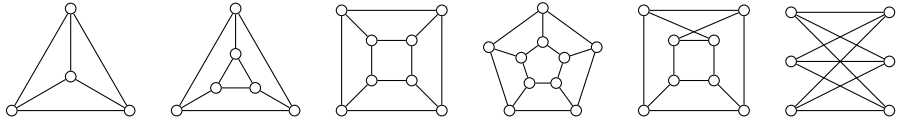

The $3$-regular graphs with nonnegative Ollivier curvature were classified by Cushing \textit{et al.}~\cite{CKLLS22} in the paper where they introduced the graph curvature calculator. They showed that the $3$-regular graphs with nonnegative Ollivier curvature are the prism graphs and the M\"obius ladders. The statement of Theorem~\ref{alld3} is that six small examples of these graphs are the only $3$-regular graphs with positive LLY curvature. After writing our proof of Theorem~\ref{alld3}, we learned that Hehl~\cite[Theorem 6.16]{H2024} independently showed that these six graphs are the only $3$-regular graphs with positive LLY curvature. Hehl's proof makes use of Cushing \textit{et al.}'s classification of $3$-regular graphs with nonnegative Ollivier curvature, whereas our proof of Theorem~\ref{alld3} only uses properties of LLY curvature. 

Interestingly, there seem to be many $3$-regular graphs with nonnegative LLY curvature which are not prisms or M\"obius ladders. 

\subsection{Improved diameter bound for positively curved $d$-regular graphs}
 Lin, Lu and Yau~\cite{LLY2011} proved that the diameter of any positively curved $d$-regular graph is at most $2d$ as a consequence of a discrete version of the Bonnet-Myers theorem for graphs with positive LLY curvature. This was the best known general upper bound on the diameter of a positively curved $d$-regular graph prior to this present work. Huang, Liu and Xia~\cite{HLX24} have proved sharper diameter upper bounds for amply regular graphs. 
 
 We can improve the general upper bound on the diameter of a positively curved $d$-regular graph by considering the curvatures of individual edges along a diameter path.  

\begin{theorem}\label{thm:improvdiambd}
Let $G$ be a positively curved $d$-regular graph. Then, 
\[\diam(G) \le 2d-2.\]
\end{theorem}

The proof of Theorem~\ref{thm:improvdiambd} is given in Section 4.  

\subsection{Positively curved $d$-regular graphs have high connectivity}

Chen, Liu and You~\cite{CLY2025} recently initiated the systematic study of the connectivity of graphs with positive LLY curvature. They proved that the vertex connectivity of a graph is bounded below by the product of its minimum degree and LLY curvature. We improve on this result for regular graphs by showing that $d$-regular positively curved graphs for $d\geq 3$ must be $3$-connected. 

\begin{theorem}\label{thm:3conn}
For each integer $d\ge 3$, every $d$-regular positively curved graph is $3$-connected.
\end{theorem}

The proof of Theorem~\ref{thm:3conn} is given in Section 5. 

\subsection{Positively curved regular planar graphs}

In \cite{LU-WANG2020}, Lu and Wang initiated the study of the size of planar graphs with positive Lin--Lu--Yau curvature. Lu and Wang showed that there are only a finite number of positively curved planar graphs with minimum degree at least $3$; in particular, they showed that any such graph has at most $17^{544}$ vertices, although it is conjectured that this upper bound can be greatly reduced. Brooks \textit{et al.}~\cite{BOSWY2025} showed that any positively curved maximal outerplanar graph with minimum degree at least $2$ has at most $10$ vertices. Liu, Lu and Wang~\cite{LLW2024} improved this result by showing that any positively curved outerplanar graph with minimum degree at least $2$ has at most $10$ vertices. The Halin graphs with positive LLY curvature have also been classified recently~\cite{CLLY26}. 

Note that locally finite planar graphs with degree at least $3$ with positive \emph{combinatorial curvature} have been well-studied. Higuchi~\cite{H2001} proposed that all such graphs must in fact be finite. This conjecture was proved by DeVos and Mohar~\cite{DM2007}, who showed that all planar graphs with positive combinatorial curvature must have at most $3444$ vertices, apart from the prism and antiprism graphs. After many improvements, Ghidelli~\cite{G2023} recently proved the optimal bound that all planar graphs with positive combinatorial curvature other than the prism and antiprism graphs must have at most $208$ vertices. 

In a 2022 REU program, Linyuan Lu, Joshua Thompson, and Leah Mangono first considered the problem of classifying positively curved $d$-regular simple planar graphs. The connected $2$-regular planar graphs are the cycles and it has been shown~\cite{LLY2011} that the only cycle graphs which are positively curved are $C_3$, $C_4$ and $C_5$. Since all $d$-regular planar graphs have a vertex of degree at most $5$, the only remaining cases are $d=3,4,5$.

The case $d=3$ is an immediate corollary of Theorem~\ref{alld3}. 

\begin{theorem}\label{d3}
    There are only four $3$-regular positively curved planar graphs: 
    $K_4$, $C_3\square P_2$, $C_4\square P_2$, and $C_5\square P_2$.
\end{theorem}

Thompson and Mangono actually classified all $3$-regular positively curved simple planar graphs. It was also verified in a computer program by Lu. However, the proof was never written down rigorously.

We now list all positively curved $4$-regular planar graphs. 

\begin{theorem}\label{d4}
    There are only five $4$-regular positively curved planar graphs: 
  the octahedron, the square antiprism, the rectified triangular prism, the cuboctahedron,
 and the pentagonal antiprism. 
\end{theorem}

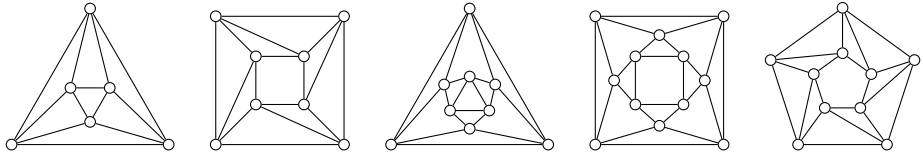
\begin{figure}[h]
\begin{center}
    \resizetikz{0.6}{\begin{tikzpicture}[scale = 1]
\foreach \k in {0,1,2}{
\node (\k) at (30 + 120*\k:0.5) {};
}
\foreach \k in {3,4,5}{
\pgfmathtruncatemacro{\i}{\k - 3}
\node (\k) at (90 + 120*\i:2) {};
}
\draw (0) -- (1) -- (2) -- (0);
\draw (3) -- (4) -- (5) -- (3);
\draw (0) -- (3) -- (1) -- (4) -- (2) -- (5) -- (0);
\end{tikzpicture}}
    \resizetikz{0.6}{\begin{tikzpicture}[scale=1]
\foreach \k in {0,...,7}{
    \ifnum \k<4
        \node (\k) at ({45+90*\k}:0.75) {};
    \else
        \node (\k) at ({45+90*\k}:2) {};
    \fi
}
\draw (0) -- (1);
\draw (1) -- (2);
\draw (2) -- (3);
\draw (3) -- (0);
\draw (4) -- (5);
\draw (5) -- (6);
\draw (6) -- (7);
\draw (7) -- (4);
\draw (0) -- (4);
\draw (1) -- (5);
\draw (2) -- (6);
\draw (3) -- (7);
\draw (0) -- (5);
\draw (1) -- (6);
\draw (2) -- (7);
\draw (3) -- (4);
\end{tikzpicture}}
    \resizetikz{0.6}{\begin{tikzpicture}[scale = 1]
\foreach \k in {3,4,5}{
\pgfmathtruncatemacro{\i}{\k - 3}
\node (\k) at (90 + 120*\i:0.5) {};
}
\foreach \k in {6,7,8}{
\pgfmathtruncatemacro{\i}{\k - 6}
\node (\k) at (30 + 120*\i:0.65) {};
}
\foreach \k in {9,10,11}{
\pgfmathtruncatemacro{\i}{\k - 9}
\node (\k) at (90 + 120*\i:2) {};
}
\draw (5) -- (3) -- (4) -- (5);
\draw (5) -- (6) -- (3) -- (7) -- (4) -- (8) -- (5);
\draw (9) -- (10) -- (11) -- (9);
\draw (9) -- (7) -- (10) -- (8) -- (11) -- (6) -- (9);
\end{tikzpicture}}
    \resizetikz{0.6}{\begin{tikzpicture}[scale = 1]
\foreach \k in {0,1,2,3}{
\node (\k) at (45 + 90*\k:0.75) {};
}
\foreach \k in {4,5,6,7}{
\pgfmathtruncatemacro{\i}{\k - 4}
\node (\k) at (90*\i:1) {};
}
\foreach \k in {8,9,10,11}{
\pgfmathtruncatemacro{\i}{\k - 8}
\node (\k) at (45 + 90*\i:2) {};
}
\draw (0) -- (1) -- (2) -- (3) -- (0);
\draw (4) -- (0) -- (5) -- (1) -- (6) -- (2) -- (7) -- (3) -- (4);
\draw (4) -- (8) -- (5) -- (9) -- (6) -- (10) -- (7) -- (11) -- (4);
\draw (8) -- (9) -- (10) -- (11) -- (8);
\end{tikzpicture}}
    \resizetikz{1}{\begin{tikzpicture}[scale=1]
\foreach \k in {0,...,9}{
    \ifnum \k<5
        \node (\k) at ({90+72*\k}:0.4) {};
    \else
        \node (\k) at ({90+72*\k}:1) {};
    \fi
}
\draw (0) -- (1);
\draw (1) -- (2);
\draw (2) -- (3);
\draw (3) -- (4);
\draw (4) -- (0);
\draw (5) -- (6);
\draw (6) -- (7);
\draw (7) -- (8);
\draw (8) -- (9);
\draw (9) -- (5);
\draw (0) -- (5);
\draw (1) -- (6);
\draw (2) -- (7);
\draw (3) -- (8);
\draw (4) -- (9);
\draw (5) -- (4);
\draw (6) -- (0);
\draw (7) -- (1);
\draw (8) -- (2);
\draw (9) -- (3);
\end{tikzpicture}}
\end{center}
\caption{The five positively curved $4$-regular planar graphs.}
\label{fig:positive4regular}
\end{figure}
Thompson and Mangono found these $4$-regular positively curved simple planar graphs. It was also verified by an exhaustive search in a computer program written by Lu. No complete analysis was written. The proof of Theorem~\ref{d4} uses Theorem~\ref{thm:3conn} and some additional arguments and is given in Section 6.  

We complete the classification of the positively curved regular planar graphs by showing that the only $5$-regular positively curved planar graph is the icosahedral graph (see Figure~\ref{fig:icosahedron}). The proof is given in Section 7.  

\begin{figure}[ht]
\centering
\begin{tikzpicture}[scale = 1]
\foreach \k in {0,1,2}{
\node (\k) at (30 + 120*\k:0.25) {};
}
\foreach \k in {3,4,5}{
\pgfmathtruncatemacro{\i}{\k - 3}
\node (\k) at (90 + 120*\i:0.5) {};
}
\foreach \k in {6,7,8}{
\pgfmathtruncatemacro{\i}{\k - 6}
\node (\k) at (30 + 120*\i:0.75) {};
}
\foreach \k in {9,10,11}{
\pgfmathtruncatemacro{\i}{\k - 9}
\node (\k) at (90 + 120*\i:2.5) {};
}
\draw (0) -- (1) -- (2) -- (0);
\draw (5) -- (0) -- (3) -- (1) -- (4) -- (2) -- (5);
\draw (5) -- (6) -- (3) -- (7) -- (4) -- (8) -- (5);
\draw (0) -- (6);
\draw (1) -- (7);
\draw (2) -- (8);
\draw (3) -- (9);
\draw (4) -- (10);
\draw (5) -- (11);
\draw (9) -- (10) -- (11) -- (9);
\draw (9) -- (7) -- (10) -- (8) -- (11) -- (6) -- (9);
\end{tikzpicture}
\caption{The icosahedral graph}
\label{fig:icosahedron}
\end{figure}
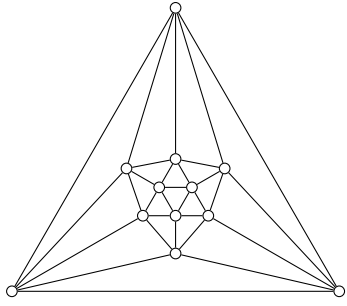

\begin{theorem}\label{d5}
    There is only one $5$-regular positively curved planar graph, the icosahedral graph. 
\end{theorem} 

We conclude in Section 8 with further research directions. 

\section{Background}

\subsection{Graph Theory Definitions}

A \emph{graph} $G = (V, E)$ is a pair consisting of a set of vertices $V$ and a set of edges $E$. Each edge $e = uv$ is a distinct $2$-element subset of $V$. For a set of vertices $S\subseteq V$, $G[S]$ denotes the \emph{induced subgraph} of $G$ on the set $S$. A graph $H$ is a \emph{minor} of a graph $G$ if $H$ can be obtained from $G$ by a sequence of edge contractions and vertex deletions. The graph $G$ is \emph{planar} if it can be drawn in the plane with no crossing edges; more formally, $G$ is planar if it is both $K_{5}$-minor-free and $K_{3, 3}$-minor-free. 

Let $u, v$ be two distinct vertices in $G$. A \emph{path} between $u$ and $v$ is a set of distinct vertices $u, u_0, \ldots u_i, v \in V$ such that $uu_0, u_0u_1, \ldots u_iv\in E$. The \emph{length} of the path $uu_0\ldots v$ is the number of edges in the path. The \emph{distance} $d(u, v)$ is the length of the shortest path between the two vertices $u$ and $v$. We assume that the graph $G$ is connected, so that there is a path between every pair of distinct vertices $u$ and $v$. The \emph{diameter} of $G$ is the longest distance between any two vertices in $G$, that is, $\text{diam}(G) = \max_{u\neq v\in V(G)}\{d(u, v)\}$. A \emph{diameter path} is a path in $G$ with length equal to the diameter of $G$ and such that there is no shorter path connecting the two endpoints of the path.  


The \emph{neighborhood} $N(v)$ of a vertex $v\in V$ is the set $N(v):= \{u\in V: uv\in E\}$. The \emph{closed neighborhood} $N[v]$ is the set $\{v\} \cup N(v)$. Let $S\subseteq V$ be a vertex subset. The \emph{first neighborhood} of $S$ is the set $N^1(S):= (\bigcup_{v\in S} N(v)) \backslash S$ of all vertices in $V(G) \setminus S$ adjacent to at least one vertex in $S$, and for each $i\ge 2$, the \emph{$i$th neighborhood} of $S$ is the set $N^{i}(S)=N^1(N^{i-1}(S))\backslash (S \cup \bigcup_{j<i} N^j(S))$. For a subgraph $H$ of $G$, let $N_H^i(S)$ denote the set $N^i(S)\cap V(H)$. Throughout, we drop the superscript from the first neighborhood of $S$ when it is clear from the context.

\subsection{Diameter Path Definitions and Lemmas}

Let $x$ and $y$ be two vertices in a graph $G$. We may divide the neighborhood of $y$ into three disjoint sets depending on the vertex's distance from $x$:
\[N_x^{+}(y) := \{v : v\in N(y),\, d(x, v) = d(x, y) + 1\} ;\]
\[N_x^{0}(y) := \{v: v\in N(y),\, d(x, v) = d(x, y)\};\]
\[N_x^{-}(y) := \{v: v\in N(y),\, d(x, v) = d(x, y) - 1\}. \]

Let $G$ be a connected $d$-regular graph with diameter $L$ and let $P = x_0x_1\dotsm x_L$ be a diameter path in $G$. We define the \textit{type} $(\star_1,\dots, \star_{d-2})$ of each vertex $x_i\in V(P)\setminus\{x_0, x_L\}$ such that for all $d-2$ neighbors $v_i^{(j)} \in  N(x_i)\setminus \{x_{i-1}, x_{i+1}\}$, $\star_j$ is the \textit{sign} of the quantity $d(x_i, x_L) - d(v_i^{(j)}, x_L) $, so that $\star_j \in \{-, 0, +\}$. We sometimes use the notation $(-^{a}, 0^b, +^{c})$ to indicate that a vertex has $a$ neighbors with sign $-$, $b$ neighbors with sign $0$ and $c$ neighbors with sign $+$. 

For example, consider the $4$-regular graph $G$ with diameter $3$ shown in Figure~\ref{fig:4Rgraph1}. In the diameter path $P = x_0x_1x_2x_3$, $x_1$ is of type $(0,0)$ and $x_2$ is of type $(0,-)$.

\begin{figure}[ht]
\centering
\begin{tikzpicture}[scale = 1.5]
\foreach \k in {0,1,2,3}{
\coordinate (\k) at (45 + 90*\k:0.75) {};
}
\foreach \k in {4,5,6,7}{
\pgfmathtruncatemacro{\i}{\k - 4}
\coordinate (\k) at (90*\i:1) {};
}
\foreach \k in {8,9,10,11}{
\pgfmathtruncatemacro{\i}{\k - 8}
\coordinate (\k) at (45 + 90*\i:2) {};
}
\draw (1) -- (2) -- (3) -- (0);
\draw (4) -- (0) -- (5) -- (1);
\draw (6) -- (2) -- (7) -- (3) -- (4);
\draw (4) -- (8) -- (5) -- (9) -- (6);
\draw (10) -- (7) -- (11) -- (4);
\draw (8) -- (9) -- (10) -- (11) -- (8);
\draw[thick] (10) -- (6) -- (1) -- (0);

\foreach \name in {2,3,4,5,7,8,9,11}{
\node at (\name) {};
}

\node[label=below:$x_0$] at (225:2) {};
\node[label=180:$x_1$] at (180:1) {};
\node[label=135:$x_2$] at (135:0.75) {};
\node[label=45:$x_3$] at (45:0.75) {};
\end{tikzpicture}
\caption{The cuboctahedron with a diameter path shown}
\label{fig:4Rgraph1}
\end{figure}
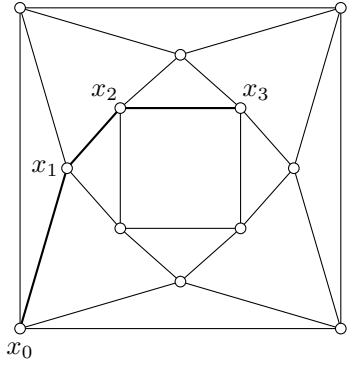

We record some basic properties of signs of vertices that will be used repeatedly in Section 4.  

\begin{lemma}\label{lem:signsnbrs}
Let $x_0 - ... - x_L$ be a diameter path in a $d$-regular graph and let $1\le i\le L-1$. Let $v_i \in N(x_i)$. Then, the following claims hold. 
\begin{enumerate}
    \item[(1)] If $v_i$ has sign $-$ with respect to $x_i$, then $v_i \notin N(x_{i+1})$. 
    \item[(2)] If $v_i$ has sign $0$ with respect to $x_i$ and $v_i \in N(x_{i+1})$, then $v_i$ has sign $-$ with respect to the vertex $x_{i+1}$. 
    \item[(3)] If $v_i$ has sign $+$ with respect to $x_i$ and $v_i \in N(x_{i+1})$, then $v_i$ has sign $0$ with respect to the vertex $x_{i+1}$. 
\end{enumerate}
\end{lemma}

\begin{proof}[Proof of Lemma~\ref{lem:signsnbrs}]
If $v_i$ has sign $-$ with respect to $x_i$, then $d(v_i, x_L) = d(x_i, x_L) + 1 = d(x_{i+1}, x_L) + 2$, so $v_i \notin N(x_{i+1})$, as otherwise there would be a shorter path from $x_L$ to $v_i$. This proves (1). 

If $v_i$ has sign $0$ with respect to $x_i$ and $v_i \in N(x_{i+1})$, then $d(v_i, x_L) = d(x_i, x_L) = d(x_{i+1}, x_L) + 1$, so $v_i$ has sign $-$ with respect to $x_{i+1}$, which proves (2). The proof of (3) is similar. 
\end{proof}

\subsection{LLY Curvature Definitions and Lemmas}

Given a simple connected graph $G = (V, E)$, the Lin-Lu-Yau (LLY) curvature can be defined between any pair of distinct vertices $x, y\in V(G)$. The original definition of Lin, Lu and Yau was given as a limit of Ollivier-Ricci curvature. We give a limit-free definition of LLY curvature in terms of the graph Laplacian, which was shown to be equivalent to the original definition of LLY curvature by M\"{u}nch and Wojciechowski~\cite{MW2019}. 

 A function $f:V(G) \rightarrow \mathbb{R}$ is \emph{$1$-Lipschitz} if $|f(x) - f(y)| \le d(x, y)$ for all $x, y \in V(G)$, where $d(x, y)$ is the graph distance between $x$ and $y$. The \emph{graph Laplacian} for the graph $G$ is defined by 
\[\Delta f(x) = \frac{1}{d_x}\sum_{v\in N(x)}(f(v)-f(x)).\]

\begin{definition}[LLY curvature~\cite{MW2019}]\label{defn:llycurv}
    For a simple connected graph $G = (V, E)$, for any distinct pair of vertices $x, y\in V(G)$, we define the LLY curvature $\kappa(x, y)$ by 
    \[\kappa(x, y) = \inf_{\nabla_{yx} f = 1} \nabla_{xy}\Delta f,\]
    where $\nabla_{xy}f= \frac{f(x) - f(y)}{d(x, y)}$ is the \emph{gradient} function. Furthermore, it suffices to optimize over integer-valued $1$-Lipschitz functions $f$. 
\end{definition}

Although LLY curvature is defined over all pairs of distinct vertices $x, y\in V(G)$, the following lemma implies that when studying positively curved graphs it suffices to show $\kappa(x, y) > 0$ when $xy\in E(G)$.

\begin{lemma}[\cite{LLY2011}]\label{lem:edgessuffice}
If $\kappa(x, y) \ge \kappa_0$ for any $xy \in E(G)$, then $\kappa(x, y) \ge \kappa_0$ for any pair of vertices $(x, y)$. 
\end{lemma}

Li and Lu~\cite{LL2024} gave an alternative formulation of LLY curvature $\kappa(x, y)$ of an edge $xy$ in a graph $G$ in terms of certain well-chosen mass distributions on $x$ and $y$. Since all the graphs we consider in this paper are $d$-regular, we simplify their formulation somewhat. 

For an edge $xy$ in a $d$-regular graph, consider the following pair of mass distributions:
\[\mu_x(v) = \begin{cases} 1 & \text{if } v\in N(x)\setminus N[y]\\ 0 & \text{otherwise}\end{cases} \quad \text{ and } \quad \mu_y(v) = \begin{cases} 1 & \text{if } v\in N(y)\setminus N[x]\\ 0 & \text{otherwise}.\end{cases} \]

A \emph{coupling} $\sigma$ between $\mu_x$ and $\mu_y$ is a map $\sigma:V(G) \times V(G) \rightarrow [0, 1]$ such that
\[\sum_{v\in V(G)} \sigma(u, v) = \mu_x(u) \text{ and } \sum_{u\in V(G)}\sigma(u, v) = \mu_y(v).\]

For any coupling $\sigma$ between $\mu_x$ and $\mu_y$, the \emph{cost function} $C_{\sigma}$ is defined by 
\[C_{\sigma} = \sum_{u, v\in V(G)}\sigma(u, v)d(u, v).\]

Li and Lu proved the following result on $\kappa(x, y)$. 

\begin{theorem}[Li--Lu~\cite{LL2024}]\label{thm:liluthm}
For any edge $xy$ in a $d$-regular graph $G$, we have 
\[\kappa(x, y) = 1 + \frac{1}{d} - \frac{\min_{\sigma}C_{\sigma}}{d},\]
where the minimum is taken over all integer-valued couplings $\sigma$ between $\mu_x$ and $\mu_y$. 
\end{theorem}

Theorem~\ref{thm:liluthm} has the very useful implication that for any edge $xy$ of a $d$-regular graph $G$, there must be a coupling $\sigma$ between $\mu_x$ and $\mu_y$ which has cost at most $d$. In particular, we have $\kappa(x, y) > 0$ if and only if there is a bijection $\phi: N(x)\setminus N[y] \rightarrow N(y)\setminus N[x]$ such that $\sum_{v\in N(x) \setminus N[y]}d(v, \phi(v)) \le d$. This characterization was independently proved by Hehl~\cite[Theorem 4.3]{H2024} (see also \cite[Theorem 2.3]{CLY2025}).  

We recall a useful lemma which provides a bound on LLY curvature if an edge is not contained in a $3$-cycle or a $4$-cycle.  

\begin{lemma}[Lin--Lu--Yau \cite{LLY2014}]\label{lem:c3c4}
    Suppose that an edge $xy$ in a graph $G$ is not in a $C_3$ or a $C_4$. Then,
    \[\kappa(x, y) \le \frac{1}{d_x} + \frac{2}{d_y} - 1.\]
\end{lemma}

Lemma~\ref{lem:c3c4} implies that for $d\ge 3$ every edge in a positively curved $d$-regular graph is contained in a $3$-cycle or a $4$-cycle. 

\begin{lemma}\label{lem:dregplanc3c4}
Every edge in a positively curved $d$-regular graph with $3\le d$ is contained in either a $C_3$ or a $C_4$.
\end{lemma}

\begin{proof}[Proof of Lemma~\ref{lem:dregplanc3c4}]
If an edge $xy$ in a positively curved $d$-regular graph is not contained in a $C_3$ or a $C_4$, then by Lemma~\ref{lem:c3c4}, 
\[\kappa(x, y) \le \frac{1}{d} + \frac{2}{d} -1 = \frac{3}{d} - 1 \le 0,\]
which is a contradiction. 
\end{proof}

\section{The positively curved $3$-regular graphs}

In this section, we prove Theorems~\ref{alld3} and \ref{d3} simultaneously by an exhaustive analysis of the possible local configurations in a positively curved $3$-regular graph. Suppose $G$ is a simple, $3$-regular positively curved graph.

\subsection*{Case 1: $K_4$}

Suppose there exists $ab \in E(G)$ such that $|N(a)\cap N(b)| = 2$ and let $N(a)\cap N(b) = \{c, d\}$. If $cd\in E(G)$, then $G \cong G[\{a,b,c,d\}] \cong K_4$ since $G$ is connected and 3-regular. The LLY curvature of each edge of $K_4$ is $\kappa_{LLY}(K_4) = 1 + \frac{1}{3} = \frac{4}{3} > 0$. In this case, $G \cong K_4$ is positively curved. 

Toward a contradiction, suppose that $cd \notin E(G)$, so that instead $ce\in E(G)$ for some new vertex $e$. If $de \notin E(G)$, then the edge $ce$ is not contained in any $C_3$ or $C_4$, so by Lemma~\ref{lem:dregplanc3c4} it follows that $\kappa(c, e) \le 0$. Now suppose that $de\in E(G)$. Since $G$ is $3$-regular, $e$ is adjacent to some new vertex $f$ which is not adjacent to any of $a, b, c, d$. Consider the two mass distributions \[\mu_c(v) = \begin{cases} 1 & \text{if } v\in \{a, b\}\\ 0 &\text{otherwise}\end{cases} \text{ and } \mu_e(v) = \begin{cases} 1 & \text{if } v\in \{d, f\}\\ 0 &\text{otherwise.}\end{cases}\] Any coupling between $\mu_c$ and $\mu_e$ must have cost at least $4$, so by Theorem~\ref{thm:liluthm}, we have that $\kappa(c, e) \le 1 + \frac{1}{3} - \frac{4}{3} = 0$, a contradiction of $G$ being positively curved.

Therefore, $G \cong K_4$.

\subsection*{Case 2: $C_3 \square P_2$}

Suppose $|N(x)\cap N(y)| \leq 1$ for all $xy\in E(G)$ and there exists $H\subseteq G$ such that $H \cong K_3$. Let $V(H)= \{a,b,c\}$ and $d, e, f$ be the distinct vertices such that $ad, be, cf \in E(G)$. We consider two subcases.

\textbf{Subcase 2.1:} Suppose $G[\{d, e, f\}] \not\cong K_3$. Then at least one of the vertices $d, e, f$ is adjacent to a vertex $g\in V(G)\setminus\{a,b,c,d,e,f\}$. Without loss of generality, assume $dg\in E(G)$. Note that if $de\notin E(G)$ and $df \notin E(G)$, then the edge $ad$ would not be contained in any $3$-cycle or $4$-cycle, so by Lemma \ref{lem:dregplanc3c4} the graph $G$ would not be positively curved. Hence, without loss of generality, we may suppose $de \in E(G)$. We must also have $fg \in E(G)$, as otherwise any coupling between $\mu_a$ and $\mu_d$ will cost at least $4$, implying that $\kappa(a, d) \le 0$ by Theorem~\ref{thm:liluthm}. Refer to Figure~\ref{fig:case2fig} for the remainder of the proof of Subcase 2.1. 

\begin{figure}[hbt]
\centering
\begin{tikzpicture}[scale=1]

\node[label=150:$a$] (0) at (90:0.5) {};
\node[label=270:$b$] (1) at (210:0.5) {};
\node[label=270:$c$] (2) at (330:0.5) {};
\node[label=90:$d$] (3) at (90:1.5) {};
\node[label=210:$e$] (4) at (210:1.5) {};
\node[label=330:$f$] (5) at (330:1.5) {};
\node[label=30:$g$] (6) at (30:0.75) {};
\draw (0) -- (1);
\draw (0) -- (2);
\draw (0) -- (3);
\draw (1) -- (2);
\draw (1) -- (4);
\draw (2) -- (5);
\draw (3) -- (4);
\draw (3) -- (6);
\draw (5) -- (6);

\end{tikzpicture}
\caption{Subcase 2.1.}
\label{fig:case2fig}
\end{figure}
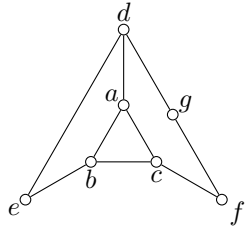

If $ef \in E(G)$, then all vertices in $\{a,b,c,d,e,f\}$ are of degree 3. It follows that $gh \in E(G)$ for some vertex $h\in V(G)\setminus\{a,b,c,d,e,f,g\}$. Therefore, any coupling between $\mu_g$ and $\mu_h$ has cost at least 6, implying $\kappa(g,h) \leq 0$ by Theorem~\ref{thm:liluthm}. Hence, $ef \not \in E(G)$. Since $|N(c) \cap N(f)| = 0$ and $|E(N(c), N(f))| = 0$, the edge $cf$ is not contained in a $3$-cycle or $4$-cycle, so by Lemma \ref{lem:dregplanc3c4} the graph $G$ would not be positively curved. Hence, if $G[\{d,e,f\}] \not \cong K_3$, every possibility leads to a contradiction of $G$ being positively curved.

\textbf{Subcase 2.2:} Suppose $G[\{d, e, f\}]\cong K_3$. Then $G\cong C_3 \square P_2$ since $G$ is 3-regular and connected. The LLY curvature is computed for $C_3 \square P_2$ as follows: 

\[ \kappa_{LLY}(x,y) = 1 + \frac{1}{3} -\frac{2}{3}= \frac{2}{3} > 0, \] for the edges between triangles

\[ \kappa_{LLY}(x,y) = 1 + \frac{1}{3} -\frac{1}{3}= 1 > 0, \] for the edges on the triangles $abc$ and $def$. Thus, the only positively curved graph in this case is $C_3\square P_2$.

Henceforth, we may assume that the graph $G$ is triangle-free, so by  Lemma~\ref{lem:dregplanc3c4} every edge in $G$ is contained in a $4$-cycle. 

\subsection*{Case 3: $C_4\square P_2$ and $K_{3,3}$}

 By Theorem~\ref{thm:liluthm}, every edge $xy\in E(G)$ satisfies $\kappa(x, y) \ge \frac13$ in a positively curved $3$-regular graph $G$. This lower bound can be improved by a more careful analysis if it is additionally assumed that $G$ does not contain a $3$-cycle or a $5$-cycle. 

\begin{lemma}\label{lem:3regnoc3noc5}
Let $G$ be a positively curved $3$-regular graph which contains no $C_3$ or $C_5$. Then, for any edge $xy \in E(G)$, we have that $\kappa(x, y) = \frac23$. 
\end{lemma}

\begin{proof}[Proof of Lemma~\ref{lem:3regnoc3noc5}]
Write $N(x)\setminus\{y\} = \{a_1, a_2\}$ and $N(y)\setminus\{x\} = \{b_1, b_2\}$. Note that for $i, j\in \{1, 2\}$, we have that $d(a_i, b_j) \ge 1$, as other wise $xa_iy$ would form a $3$-cycle, and also $d(a_i, b_j) \neq 2$, as otherwise $x$, $y$ and the vertices of a length-two $a_ib_i$-path would form a $5$-cycle. These observations imply that an integer-valued coupling between $\mu_x$ and $\mu_y$ can have cost at most $3$ only if, say, $d(a_1, b_1) = 1$ and $d(a_2, b_2) = 1$.  Theorem~\ref{thm:liluthm} then implies that 
\[\kappa(x, y) =1 + \frac13 - \frac23 = \frac23.\]
\end{proof}

We now recall the following theorem of Gamlath, Liu, Lu and Yuan. 

\begin{theorem}[\cite{GLLY2023}]\label{thm:glly}
Let $G$ be a simple connected graph on $n$ vertices containing no $C_3$ or $C_5$. If for every edge $xy\in E(G)$, we have $\kappa(x, y) \ge \kappa > 0$, then 
\[n \le 2^{2/\kappa}.\]
\end{theorem}

Suppose $G$ does not contain a $3$-cycle or a $5$-cycle. Lemma~\ref{lem:3regnoc3noc5} and Theorem~\ref{thm:glly} imply that the number of vertices in $G$ is at most $2^{2/(2/3)} = 8$. One can check directly by an enumeration of the $3$-regular planar graphs on at most $8$ vertices that the only such graphs which are positively curved and contain no $3$-cycle or $5$-cycle are the complete bipartite graph $K_{3, 3}$ and the cube $C_4\square P_2$. The cube is planar, while $K_{3, 3}$ is not.

\subsection*{Case 4: $C_5\square P_2$ and the Wagner graph}

Suppose $G$ does not contain a $3$-cycle but does contain a $5$-cycle. Let $C =\{a_1,a_2,a_3,a_4,a_5\}$ be a set of vertices that make up a $5$-cycle in $G$ where $a_ia_{i+1}\in E(G)$ for $1\le i\le 4$ and $a_5a_1\in E(G)$. As $G$ is triangle-free, none of the interior chords in the $5$-cycle can be edges in the graph, so each of the vertices of the $5$-cycle is adjacent to some vertex in $V(G) \setminus C$. For any $x \in V(G)\setminus C$, we have $|N(x) \cap C| \leq 2$ since $G$ is triangle-free. It follows that $|N(C)| \in \{3, 4, 5\}$. We consider each subcase.

{\bf Subcase 4.1: $|N(C)| =5 $.}

Let $N(C) = \{b_1,b_2,b_3,b_4,b_5\}$ such that $a_ib_i \in E(G)$ for $1\leq i\leq 5$ as shown in Figure~\ref{fig:Case4sub1}.

Recall by Lemma~\ref{lem:dregplanc3c4} that each edge must be contained in a $4$-cycle. Therefore, the edge $a_1a_2$ must be contained in a $4$-cycle and since the $b_i$s are distinct vertices and the graph is $3$-regular, this is only possible if $b_1b_2 \in E(G)$. Similarly, we must have that $b_2b_3, b_3b_4, b_4b_5, b_5b_1 \in E(G)$. The resulting graph $G$ is isomorphic to $C_5 \square P_2$, which is positively curved, $3$-regular and planar. Thus, in this subcase we must have $G \cong C_5\square P_2$. 

\begin{figure}[ht]
\centering
\begin{tikzpicture}[scale=2]
\node[label=135:$b_1$] (0) at (90:0.4) {};
\node[label=207:$b_2$] (1) at (162:0.4) {};
\node[label=279:$b_3$] (2) at (234:0.4) {};
\node[label=0:$b_4$] (3) at (306:0.4) {};
\node[label=90:$b_5$] (4) at (18:0.4) {};
\node[label=90:$a_1$] (5) at (90:1) {};
\node[label=162:$a_2$] (6) at (162:1) {};
\node[label=234:$a_3$] (7) at (234:1) {};
\node[label=306:$a_4$] (8) at (306:1) {};
\node[label=18:$a_5$] (9) at (18:1) {};
\draw (5) -- (6);
\draw (6) -- (7);
\draw (7) -- (8);
\draw (8) -- (9);
\draw (9) -- (5);
\draw (0) -- (5);
\draw (1) -- (6);
\draw (2) -- (7);
\draw (3) -- (8);
\draw (4) -- (9);
\end{tikzpicture}
\caption{Subcase 4.1}
\label{fig:Case4sub1}
\end{figure}
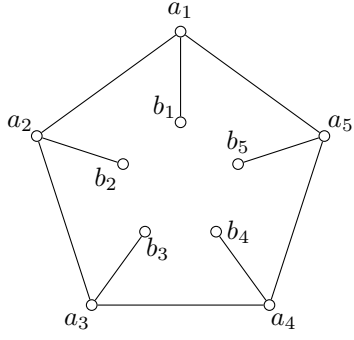

{\bf Subcase 4.2: $|N(C)| =4 $.}

Let $N(C) = \{b,c,d,e\}$ such that $a_1b, a_2c, a_3d, a_4e, a_5c\in E(G)$ as shown in Figure~\ref{fig:Case4sub2}.

We claim that there are no positively curved $3$-regular graphs in this case. Consider the edge $a_2a_3$. We have $N(a_2)\setminus\{a_3\} = \{a_1, c\}$ and $N(a_3)\setminus\{a_2\}=\{a_4, d\}$. Now, $d(a_1, d) \ge 2$, $d(a_1, a_4) = 2$ and $d(a_4, c) \ge 2$. Thus, in order for there to be an integer-valued coupling between $\mu_{a_2}$ and $\mu_{a_3}$ with cost at most $3$, we must have that $d(c, d) = 1$, \textit{i.e.}, that $cd\in E(G)$. Similarly, in order for the edge $a_4a_5$ to be positively curved, we must have that $ce\in E(G)$. But now the degree of $c$ is at least $4$, a contradiction to $G$ being $3$-regular. 

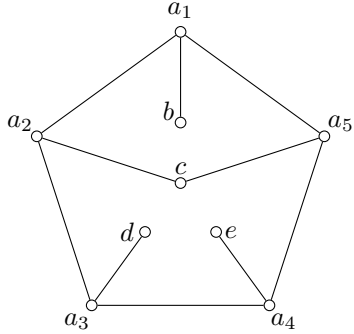
\begin{figure}[ht]
\centering
\begin{tikzpicture}[scale=2]
\node[label=135:$b$] (0) at (90:0.4) {};
\node[label=90:$c$] (1) at (0:0) {};
\node[label=180:$d$] (2) at (234:0.4) {};
\node[label=0:$e$] (3) at (306:0.4) {};
\node[label=90:$a_1$] (5) at (90:1) {};
\node[label=162:$a_2$] (6) at (162:1) {};
\node[label=234:$a_3$] (7) at (234:1) {};
\node[label=306:$a_4$] (8) at (306:1) {};
\node[label=18:$a_5$] (9) at (18:1) {};
\draw (5) -- (6);
\draw (6) -- (7);
\draw (7) -- (8);
\draw (8) -- (9);
\draw (9) -- (5);
\draw (0) -- (5);
\draw (1) -- (6);
\draw (1) -- (9);
\draw (2) -- (7);
\draw (3) -- (8);
\end{tikzpicture}
\caption{Subcase 4.2}
\label{fig:Case4sub2}
\end{figure}
{\bf Subcase 4.3: $|N(C)| =3 $.}

Let $N(C) = \{b,c,d\}$ such that $a_1b, a_2c, a_3d, a_4c, a_5d \in E(G)$ as shown in Figure~\ref{fig:Case4sub3}.

Consider the edge $a_1a_2$ in this graph. We have $N(a_1)\setminus\{a_2\} = \{a_5, b\}$ and $N(a_2)\setminus\{a_1\} = \{a_3, c\}$. We compute $d(a_3, a_5) =2$, $d(a_3, b)\ge 2$ and $d(a_5,c)\ge 2$. Therefore, in order for there to be an integer-valued coupling between $\mu_{a_1}$ and $\mu_{a_2}$ with cost at most $3$, we must have that $bc \in E(G)$. Similarly, in order for the edge $a_1a_5$ to be positively curved, we must have that $bd\in E(G)$. The resulting graph $G$ is isomorphic to the Wagner graph, which is $3$-regular, positively curved and nonplanar.
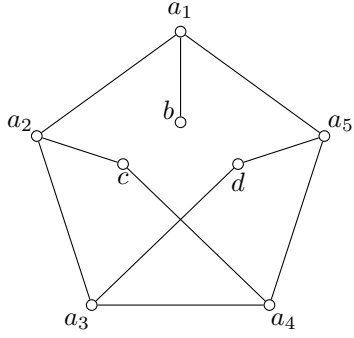
\begin{figure}[ht]
\centering
\begin{tikzpicture}[scale=2]
\node[label=135:$b$] (0) at (90:0.4) {};
\node[label=270:$c$] (1) at (162:0.4) {};
\node[label=270:$d$] (4) at (18:0.4) {};
\node[label=90:$a_1$] (5) at (90:1) {};
\node[label=162:$a_2$] (6) at (162:1) {};
\node[label=234:$a_3$] (7) at (234:1) {};
\node[label=306:$a_4$] (8) at (306:1) {};
\node[label=18:$a_5$] (9) at (18:1) {};
\draw (5) -- (6);
\draw (6) -- (7);
\draw (7) -- (8);
\draw (8) -- (9);
\draw (9) -- (5);
\draw (0) -- (5);
\draw (1) -- (6);
\draw (4) -- (9);
\draw (4) -- (7);
\draw (1) -- (8);
\end{tikzpicture}
\caption{Subcase 4.3}
\label{fig:Case4sub3}
\end{figure}

This completes the casework analysis and completely determines the positively curved $3$-regular graphs. 

\section{Diameter bounds for $d$-regular graphs}
Recall that for a positively curved $d$-regular graph, the Bonnet-Myers theorem for LLY curvature implies $\diam(G) \le 2d$. In this section, we improve this general bound to $\diam(G) \le 2d-2$. 

Suppose $G$ is a $d$-regular graph with diameter $L$ and let $P = x_0x_1\dotsm x_L$ be a diameter path in $G$. We first consider the case that two consecutive internal vertices along the path do not have any common neighbors. Note that in the following two lemmas we do not assume that $G$ is positively curved. 

\begin{lemma}\label{lem:dregedgenocmnnbhrs}
Let $1\le i\le L-2$. Suppose that the vertex $x_i$ is of type $(-^{a_1}, 0^{b_1}, +^{c_1})$ and that the vertex $x_{i+1}$ is of type $(-^{a_2}, 0^{b_2}, +^{c_2})$. Let $s\ge 0$ be a nonnegative integer and suppose that $c_1 - a_1 \le c_2 - a_2 - s$ and that $x_i$ and $x_{i+1}$ do not have any common neighbors. Then, for any bijection $\phi:N(x_i)\setminus\{x_{i+1}\} \rightarrow N(x_{i+1})\setminus\{x_i\}$, we have that $\sum_{v\in N(x_i)\setminus\{x_{i+1}\}}d(v, \phi(v)) \ge d+1+s. $
\end{lemma}

\begin{proof}[Proof of Lemma~\ref{lem:dregedgenocmnnbhrs}]
Denote the set of neighbors of $x_i$ with sign $-$ by $N_i$, the set of neighbors of $x_i$ with sign $0$ by $Z_i$ and the set of neighbors of $x_i$ with sign $+$ by $P_i$, so that $|N_i| = a_1$, $|Z_i|=b_1$ and $|P_i|=c_1$. The sets $N_{i+1}, Z_{i+1}, P_{i+1}$ are defined analogously for the vertex $x_{i+1}$. Since $x_i$ and $x_{i+1}$ have no common neighbors, we trivially have $\sum_{v\in N(x_i)\setminus \{x_{i+1}\}}d(v, \phi(v)) \ge d-1$ for any bijection $\phi:N(x_i)\setminus\{x_{i+1}\} \rightarrow N(x_{i+1})\setminus\{x_i\}$. We show that the assumption $c_1 - a_1 \le c_2 - a_2 - s$ leads to an improvement of this easy lower bound. 

Suppose that in the bijection $\phi$ there are exactly $k$ pairs $(u_1, v_1), \ldots (u_k, v_k)$ with $u_m\in N_i \cup \{x_{i-1}\}$ and $v_m\in P_{i+1} \cup \{x_{i+2}\}$ for $1\le m\le k$. Each of these pairs has $d(u_m, v_m)\ge 3$, so we already obtain the improved bound $\sum_v d(v, \phi(v)) \ge d-1+2k$. If $2k\ge s+2$, then it immediately follows that \[\sum_{v\in N(x_i)\setminus\{x_{i+1}\}} d(v, \phi(v)) \ge d-1+2k \ge d-1+s+2 = d+1+s.\] Now assume that $s > 2k-2$. Consider the sets of pairs $A = \{(u, \phi(u)): u\in (N_i \cup \{x_{i-1}\})\setminus \{u_1, \ldots, u_k\}\}$ and $B = \{(\phi^{-1}(v), v): v\in (P_{i+1} \cup \{x_{i+2}\}) \setminus \{v_1, \ldots v_k\}\}$. By construction, $A$ and $B$ are disjoint sets, $|A| = a_1 + 1 - k$ and $|B| = c_2 + 1 - k$. Also, for $(u, \phi(u))\in A$, we have $d(u, \phi(u)) \ge 2$ unless $\phi(u)\in N_{i+1}$, and similarly $d((\phi^{-1}(v), v)) \ge 2$ for $(\phi^{-1}(v), v) \in B$ unless $\phi^{-1}(v)\in P_i$. We may rearrange the inequality $c_1 - a_1 \le c_2 - a_2 - s$ to get 
\[((a_1 + 1 -k)-a_2) + ((c_2 + 1 -k) - c_1) \ge s-2k + 2\]

\[ \iff (|A| - |N_{i+1}|) + (|B|-|P_i|) \ge s-2k + 2.\]
In particular, this inequality implies that if $s> 2k - 2$, then there are at least $s-2k + 2$ pairs from $A\cup B$ which have distance at least $2$. It follows that 
\[\sum_{v\in N(x_i)\setminus\{x_{i+1}\}} d(v, \phi(v))\ge d-1+2k+(s-2k+2) = d+1+s,\]
completing the proof. 
\end{proof}

Using Lemma~\ref{lem:dregedgenocmnnbhrs}, we can handle the general case where $x_i$ and $x_{i+1}$ may have common neighbors. 

\begin{lemma}\label{lem:dregedge}
Let $1\le i\le L-2$. Suppose that the vertex $x_i$ is of type $(-^{a_1}, 0^{b_1}, +^{c_1})$ and that the vertex $x_{i+1}$ is of type $(-^{a_2}, 0^{b_2}, +^{c_2})$. Suppose that $c_1 - a_1 \le c_2 - a_2$. Then, for any bijection $\phi:N(x_i)\setminus N[x_{i+1}] \rightarrow N(x_{i+1}) \setminus N[x_i]$, we have $\sum_{v\in N(x_i)\setminus N[x_{i+1}]}d(v, \phi(v)) \ge d+1$. 
\end{lemma}

\begin{proof}[Proof of Lemma~\ref{lem:dregedge}]
By Lemma~\ref{lem:signsnbrs}, none of the $a_1$ neighbors of $x_i$ with sign $-$ are also neighbors of $x_{i+1}$, while the neighbors of $x_{i}$ with sign $0$ which are also neighbors of $x_{i+1}$ have sign $-$ with respect to $x_{i+1}$ and the neighbors of $x_i$ with sign $+$ which are also neighbors of $x_{i+1}$ have sign $0$ with respect to $x_{i+1}$. Suppose that $k$ of the neighbors of $x_i$ of sign $0$ are also neighbors of $x_{i+1}$ and $\ell$ of the neighbors of $x_i$ of sign $+$ are also neighbors of $x_{i+1}$. By ignoring these common neighbors, we may consider $x_i$ and $x_{i+1}$ to be vertices which have no common neighbors and each of which has degree $d - k - \ell$. Furthermore, after this vertex deletion, $x_i$ is of type $(-^{a_1}, 0^{b_1-k}, +^{c_1-\ell})$ and $x_{i+1}$ is of type $(-^{a_2-k}, 0^{b_2-\ell}, +^{c_2})$. Finally, the assumed inequality $c_1 - a_1 \le c_2-a_2$ is equivalent to $(c_1 - \ell) - a_1\le c_2 - (a_2 - k) - (k + \ell)$. Therefore, by Lemma~\ref{lem:dregedgenocmnnbhrs}, for any bijection $\phi:N(x_i)\setminus N[x_{i+1}] \rightarrow N(x_{i+1}) \setminus N[x_i]$, we have that $\sum_{v\in N(x_i)\setminus N[x_{i+1}]}d(v, \phi(v)) \ge (d-k-\ell) + 1 + (k+\ell) = d+1$.
\end{proof}

Lemma~\ref{lem:dregedge} and Theorem~\ref{thm:liluthm} immediately give the following corollary for positively curved regular graphs. 

\begin{corollary}\label{cor:dregedge}
Let $1\le i\le L-2$. Suppose that the vertex $x_i$ is of type $(-^{a_1}, 0^{b_1}, +^{c_1})$ and that the vertex $x_{i+1}$ is of type $(-^{a_2}, 0^{b_2}, +^{c_2})$. Then, if $G$ is positively curved, we have that $c_1 - a_1 > c_2 - a_2$. 
\end{corollary}

We now prove Theorem~\ref{thm:improvdiambd}. 

\begin{proof}[Proof of Theorem~\ref{thm:improvdiambd}]
Corollary~\ref{cor:dregedge} implies that along a diameter path $x_0x_1\ldots x_L$, the sequence $(c_i-a_i)_{i=1}^{L-1}$ must be strictly decreasing. Since $a_i+b_i+c_i = d-2$ and $a_i, b_i, c_i$ are nonnegative integers, it follows that $-(d-2)\le c_i - a_i \le d-2$. Therefore, there can be at most $2d-3$ internal vertices along the diameter path, implying $L\le 2d-2$. 
\end{proof}

\section{Positively curved regular graphs are $3$-connected}

In this section, we prove Theorem~\ref{thm:3conn}, which states that every positively curved $d$-regular graph is $3$-connected when $d\ge 3$. 

We first prove that every $d$-regular positively curved graph is $2$-connected.

\begin{lemma}\label{lem:2conn}
For each integer $d\ge 2$, every $d$-regular positively curved graph is $2$-connected.
\end{lemma}
\begin{proof}
Let $d\geq 2$ be an integer and $G$ be a $d$-regular positively curved graph. Suppose $G$ is connected, but not $2$-connected. Then $G$ has a cut vertex $u$. Let $C$ denote a component of $G-u$ such that $u$ has at most $\lfloor \frac{d}{2}\rfloor$ neighbors in $C$, and let $x$ denote a neighbor of $u$ in $C$. For any bijection $\phi:N(u)\setminus N[x] \rightarrow N(x) \setminus N[u]$, we claim that $\sum_{v\in N(u)\setminus N[x]}d(v, \phi(v)) \ge 3\left\lceil \frac{d}{2}\right\rceil$.

Observe that $N[x]\setminus \{u\}\subseteq V(C)$, and for any $v\in N(u)\setminus V(C)\subseteq N(u)\setminus N[x]$ we have $d(v, \phi(v))\geq 3$ as $\phi(v)\in N(x)\setminus N[u]\subseteq V(C)\setminus N[u]$. Since $u$ has at most $\lfloor \frac{d}{2}\rfloor$ neighbors in $C$, 
it follows that $|N(u)\setminus V(C)|\geq d-\lfloor \frac{d}{2}\rfloor=\lceil \frac{d}{2}\rceil$. Hence, $$\sum_{v\in N(u)\setminus N[x]}d(v, \phi(v)) \geq \sum_{v\in N(u)\setminus V(C)}d(v, \phi(v))\geq 3\left\lceil \frac{d}{2}\right\rceil.$$
As $d\ge 2$, we have that $3\lceil \frac{d}{2}\rceil\geq d+1$, and so $\sum_{v\in N(u)\setminus N[x]}d(v, \phi(v)) \geq d+1$, contradicting that the edge $ux$ has positive curvature. Therefore, $G$ is $2$-connected.
\end{proof}

We now prove that every positively curved $d$-regular graph is $3$-connected when $d\ge 3$.

\begin{proof}[Proof of Theorem~\ref{thm:3conn}]
Let $d$ be an integer such that $d\geq 3$, and let $G$ be a $d$-regular positively curved graph. Then it follows from Lemma~\ref{lem:2conn} that $G$ is $2$-connected. Suppose $G$ is not  $3$-connected. Then $G$ has a $2$-cut $S=\{u_1, u_2\}$. Let $A$ be the event defined by $A:=\{u_1u_2\in E(G)\}$, and let $\ell:=d -\mathbbm{1}_{A}$. We claim that for each $u\in S$ and for any component $C$ of $G-S$, we have $\lfloor\frac{\ell}{2} \rfloor\leq |N_C(u)|\leq \lceil \frac{\ell}{2}\rceil$. Note that $|E(\{u\}, G-S)|=\ell$ for each $u\in S$.
It suffices to show that if $|N_C(u)|\leq \lfloor\frac{\ell}{2} \rfloor$ for some component $C$ of $G-S$ then $u$ has exactly $\lfloor\frac{\ell}{2} \rfloor$ neighbors in $C$. 

Let $C_1$ denote a component of $G-S$, and let $x_1$ be a neighbor of $u_1$ in $C_1$. 
Assume that $u_1$ has at most $\lfloor\frac{\ell}{2} \rfloor$ neighbors in $C_1$. We need to show that $|N_{C_1}(u_1)|=\lfloor\frac{\ell}{2} \rfloor$. Assume that $|N_{C_1}(u_1)|\leq \lfloor\frac{\ell}{2} \rfloor-1$. It follows that $|N_{G-S-C_1}(u_1)|\geq \lceil\frac{\ell}{2} \rceil+1$. By positive curvature of the edge $u_1x_1$, choose a bijection $\phi: N(u_1)\setminus N[x_1]\rightarrow N(x_1) \setminus N[u_1]$ satisfying $\sum_{v\in N(u_1)\setminus N[x_1]} d(v,\phi(v)) \leq d$. Observe that $N(x_1)\subseteq V(C_1)\cup S$. For any $v\in N_{G-S-C_1}(u_1)\subseteq N(u_1)\setminus N[x_1]$, if $d(v, \phi(v))=1$ then we have $\phi(v)=u_2\in N(x_1)\setminus N[u_1]$  and so $u_1u_2\notin E(G).$ Thus, 
\begin{align*}
    \sum_{v\in N_{G-S-C_1}(u_1)} d(v, \phi(v)) &\geq 2| N_{G-S-C_1}(u_1)|-\mathbbm{1}_{u_2\in N(x_1)\setminus N[u_1]}\\
    &\geq 2| N_{G-S-C_1}(u_1)|+\mathbbm{1}_{u_1u_2\in E(G)}-1\\&
    \geq 2\left(\left\lceil\frac{\ell}{2} \right\rceil+1\right)+\mathbbm{1}_A-1\\&
    \geq \ell+\mathbbm{1}_A+1=d+1,
    \end{align*}
giving a contradiction. Therefore, for every $u\in S$, we have that $\lfloor\frac{\ell}{2} \rfloor\leq |N_C(u)|\leq \lceil \frac{\ell}{2}\rceil$ for every component $C$ of $G-S$. 

Let $C_1, \dots, C_k$ be the components of $G-S$. For each $u \in S$,  $\ell = \sum_{i = 1}^{k} |N_{C_i}(u)| \geq k \lfloor\frac{\ell}{2} \rfloor$. Since $k\geq 2$, this implies $G-S$ has exactly two components unless $\ell = 3$. Suppose $\ell = 3$ and $G-S$ has three components $C_1, C_2$ and $C_3$. It follows that $|N_{C_i}(u_1)| = |N_{C_i}(u_2)|= 1$ for each $i \in [3]$. Furthermore, $N(u_1) \cap N(u_2) = \emptyset$ since $G$ is $2$-connected. Assume $x_1 \in N_{C_1}(u_1)$. By positive curvature of the edge $u_1x_1$, choose a bijection $\phi: N(u_1)\setminus N[x_1] \rightarrow N(x_1)\setminus N[u_1]$ satisfying $\sum_{v\in N(u_1)\setminus N[x_1]} d(v,\phi(v)) \leq d$.  We know that $d(v, \phi(v)) \geq 3$ for each $v \in N_{C_2}(u_1) \cup N_{C_3}(u_1)$ since $N(u_1) \cap N(u_2) = \emptyset$. Hence
\[
d \geq \sum_{v\in N(u_1)\setminus N[x_1]} d(v, \phi(v)) \geq \sum_{v \in N_{C_2}(u_1) \cup N_{C_3}(u_1)} d(v, \phi(v)) \geq 6,
\]
a contradiction of $d\leq \ell + 1 = 4$. Therefore $G-S$ has exactly two components where each $u\in S$ has $\lfloor\frac{\ell}{2} \rfloor$ neighbors in one component and $\lceil \frac{\ell}{2}\rceil$ neighbors in the other component.

Without loss of generality, we may assume $u_1$ has exactly $\lfloor\frac{\ell}{2} \rfloor$ neighbors in $C_1$. We first show that $u_1u_2\notin E(G)$. Suppose $u_1$ is adjacent to $u_2$ in $G$. Then $\ell=d-\mathbbm{1}_A=d-\mathbbm{1}_{u_1u_2\in E(G)}=d-1$. We claim that $N_{C_1}(u_1)\subseteq N_{C_1}(u_2)$, i.e., $N_{C_1}(S)=N_{C_1}(u_2)$. Assume that $N_{C_1}(u_1)\nsubseteq N_{C_1}(u_2)$. 
Then we may assume that $x_1\in N_{C_1}(u_1)\backslash N_{C_1}(u_2)$.  Note that $u_2$ and all the vertices in $N_{C_2}(u_1)$ are contained in $N(u_1)\setminus N[x_1]$. Let $N_{C_2}(u_1)=\{v_1, v_2, \ldots, v_{\lceil \frac{\ell}{2}\rceil}\}$.
By positive curvature of the edge $u_1x_1$, choose a bijection $\phi: N(u_1)\setminus N[x_1] \rightarrow N(x_1)\setminus N[u_1]$ satisfying $\sum_{v\in N(u_1)\setminus N[x_1]} d(v,\phi(v)) \leq d$. Observe 
\[d\geq \sum_{v\in N(u_1)\setminus N[x_1]} d(v, \phi(v))\geq d(u_2, \phi(u_2))+\sum_{i=1}^{\lceil \frac{\ell}{2}\rceil}d(v_i, \phi(v_i)).\]
Since $d(v_i, \phi(v_i))\geq 2$ for any $i\in [\lceil \frac{\ell}{2}\rceil]$, the above inequality forces $d(u_2, \phi(u_2))=1$ and $d(v_i, \phi(v_i))=2$. This implies that $\phi(u_2), v_1, \ldots, v_{\lceil \frac{\ell}{2}\rceil}, \phi(v_1), \ldots, \phi(v_{\lceil \frac{\ell}{2}\rceil})$ are neighbors of $u_2$ and so $u_2$ has degree at least $1+2\lceil \frac{\ell}{2}\rceil+1>d$ as $u_1u_2\in E(G)$, giving a contradiction. Hence, $x_1$ is also adjacent to $u_2$ for any $x_1\in N_{C_1}(u_1)$, implying that $N_{C_1}(u_1)\subseteq N_{C_1}(u_2)$ and $N_{C_1}(S)=N_{C_1}(u_2)$ has at most $\lceil \frac{\ell}{2}\rceil=\lceil \frac{d-1}{2}\rceil$ vertices. Observe that when $d=3$, $N_{C_1}(S)$ contains exactly one vertex and it is a cut vertex of $G$, contradicting that $G$ is $2$-connected. We may now assume that $d\ge 4$, and then $\ell=d-1\geq 3$. Note that $N(x_1)\setminus N[u_1]= (N_{C_1}(x_1)\setminus N_{C_1}(S))\cup (N_{C_1}(x_1)\cap N_{C_1}(u_2)\setminus N_{C_1}(u_1))$. Observe that $|N_{C_1}(u_2)\setminus N_{C_1}(u_1)|\leq \lceil \frac{\ell}{2}\rceil-\lfloor\frac{\ell}{2} \rfloor=\mathbbm{1}_{\ell \text{ is odd}}$. For any $v_i\in N_{C_2}(u_1)$, we know that $d(v_i, \phi(v_i))\geq 3$ if $\phi(v_i)\in N_{C_1}(x_1)\setminus N_{C_1}(S)$, and $d(v_i, \phi(v_i))\geq 2$ if $\phi(v_i)\in N_{C_1}(x_1)\cap N_{C_1}(u_2)\setminus N_{C_1}(u_1)$. It follows that 
\[ \sum_{v\in N(u_1)\setminus N[x_1]} d(v, \phi(v)) \geq \sum_{i=1}^{\lceil \frac{\ell}{2}\rceil} d(v_i, \phi(v_i))\geq 3 \left\lceil \frac{\ell}{2}\right\rceil - \mathbbm{1}_{\ell \text{ is odd}}=\ell+\left\lceil \frac{\ell}{2}\right\rceil\geq \ell+2=d+1,
\]
giving a contradiction. Therefore, $u_1u_2\notin E(G)$ and $\ell=d$.

We consider first the case that $d$ is odd. It follows that $|N_{C_1}(u_1)|=\lfloor \frac{d}{2} \rfloor=\frac{d-1}{2}$ and $|N_{C_2}(u_1)|=\lceil \frac{d}{2} \rceil=\frac{d+1}{2}$. We show that $N_{C_1}(u_1)\subseteq N_{C_1}(u_2)$. Suppose there exists $x_1\in N_{C_1}(u_1)\setminus N_{C_1}(u_2)$. By positive curvature of the edge $u_1x_1$, choose a bijection $\phi: N(u_1)\setminus N[x_1] \rightarrow N(x_1)\setminus N[u_1]$ satisfying $\sum_{v\in N(u_1)\setminus N[x_1]} d(v,\phi(v)) \leq d$. For every $v \in N_{C_2}(u_1)$, we have $d(v, \phi(v))\geq 2$. Hence,
$ \sum_{v\in N(u_1)\setminus N[x_1]} d(v, \phi(v))\geq \sum_{v\in N_{C_2}(u_1)} d(v, \phi(v))\geq 2|N_{C_2}(u_1)|=d+1,$ giving a contradiction. Therefore, $u_2\in N(x_1)\setminus N[u_1]$, and since $\sum_{v\in N(u_1)\setminus N[x_1]} d(v, \phi(v)) \leq d$, the vertex $v_1=\phi^{-1}(u_2)\in N_{C_2}(u_1)$ and $v_1u_2\in E(G)$. For any $v\in N_{C_2}(u_1)\setminus \{v_1\}$, we have $d(v, \phi(v))\geq 2$. We claim that $N_{C_2}(u_1)\subseteq N_{C_2}(u_2)$. Suppose there exists $v\in N_{C_2}(u_1)$ such that $v$ is not adjacent to $u_2$. Then $v\neq v_1$ and $d(v, \phi(v))\geq 3$. This implies that $ \sum_{v\in N(u_1)\setminus N[x_1]} d(v, \phi(v))\geq d+1$. Therefore, $N_{C_2}(u_1)\subseteq N_{C_2}(u_2)$. Since $N_{C_1}(u_1)\subseteq N_{C_1}(u_2)$ and $G$ is $d$-regular, we know that $N_{C_1}(u_1)=N_{C_1}(u_2)=N_{C_1}(S)$. It follows that for any $x_1\in N_{C_1}(u_1)$, we have $N(x_1)\setminus N[u_1]\subset \{u_2\} \cup (N_{C_1}(x_1)\setminus N_{C_1}(S)).$ This implies that for $v\in N_{C_2}(u_1)$, we have $d(v, \phi(v))\geq 3$  if $\phi(v)\neq u_2$, and so
\[\sum_{v\in N(u_1)\setminus N[x_1]} d(v, \phi(v))\geq \sum_{v\in N_{C_2}(u_1)} d(v, \phi(v))\geq 3|N_{C_2}(u_1)|-2=3\frac{d+1}{2}-2=\frac{3d-1}{2}\geq d+1\]
for $d\geq 3$. This gives a contradiction.

We may now assume that $d$ is even. Then $d\geq 4$, $|N_{C_1}(u_1)|=|N_{C_1}(u_2)|=\frac{d}{2}$. We claim that $N_{C_1}(u_1)\neq N_{C_1}(u_2)$. Suppose $N_{C_1}(u_1)=N_{C_1}(u_2)$. When $d=4$, let $N_{C_1}(u_1)=N_{C_1}(u_2)=\{x_1,x_2\}=N_{C_1}(S)$, and let $N_{C_2}(u_1)=\{v_1, v_2\}$. Observe that $\{x_1, x_2\}$ is a $2$-cut of $G$ since $G$ is $4$-regular. It follows that $x_1x_2\notin E(G)$, and so $N(u_1)\setminus N[x_1]=\{v_1, v_2, x_2\}$. By positive curvature of the edge $u_1x_1$, choose a bijection $\phi: N(u_1)\setminus N[x_1] \rightarrow N(x_1)\setminus N[u_1]$ satisfying $\sum_{v\in N(u_1)\setminus N[x_1]} d(v,\phi(v)) \leq d$. We know that one of $\phi(v_1)$ and $\phi(v_2)$ is contained in $N_{C_1}(x_1)\setminus N_{C_1}(S)$, and we may assume that $\phi(v_1)\in N_{C_1}(x_1)\setminus N_{C_1}(S)$. Then $d(v_1, \phi(v_1))\geq 3$ and $\sum_{v\in N(u_1)\setminus N[x_1]} d(v, \phi(v))=d(v_1, \phi(v_1))+d(v_2, \phi(v_2))+d(x_2, \phi(x_2))\geq 3+1+1\geq d+1.$ Hence, we may assume $d\geq 6$. Note that $N(x_1)\setminus N[u_1]\subseteq \{u_2\} \cup (N_{C_1}(x_1)\setminus N_{C_1}(S))$ for every $x_1\in N_{C_1}(u_1)=N_{C_1}(u_2)$, and so for $v\in N_{C_2}(u_1)$, we have $d(v, \phi(v))\geq 3$  if $\phi(v)\neq u_2$. Therefore,
\[\sum_{v\in N(u_1)\setminus N[x_1]} d(v, \phi(v))\geq \sum_{v\in N_{C_2}(u_1)} d(v, \phi(v))\geq 3|N_{C_2}(u_1)|-2=3\frac{d}{2}-2=\frac{3d-4}{2}\geq d+1\]
for $d\geq 6$, a contradiction. Hence, $N_{C_1}(u_1)\neq N_{C_1}(u_2)$. Similarly, $N_{C_2}(u_1)\neq N_{C_2}(u_2)$. This implies that there exist $x_1\in N_{C_1}(u_1)$ and $v_1\in N_{C_2}(u_1)$ such that both $x_1, v_1$ are not adjacent to $u_2$. Observe that $u_2\notin N(x_1)\setminus N[u_1]$, and so we know $d(v, \phi(v))\geq 2$ for each $v\in N_{C_2}(u_1)$, moreover, if $vu_2\notin E(G)$ then $d(v, \phi(v))\geq 3$. Hence, $d(v_1, \phi(v_1))\geq 3$ and $\sum_{v\in N(u_1)\setminus N[x_1]} d(v, \phi(v))\geq \sum_{v\in N_{C_2}(u_1)} d(v, \phi(v))\geq 2|N_{C_2}(u_1)|+1=d+1,$ a contradiction.

Therefore, $G$ is $3$-connected if $G$ is $d$-regular and positively curved for some integer $d\geq 3$.
\end{proof}

\section{The positively curved $4$-regular planar graphs}

Let $G$ be a positively curved $4$-regular planar graph. By Theorem~\ref{thm:3conn}, $G$ is $3$-connected. We first introduce some additional notation for $3$-connected plane graphs. Let $G$ be a $3$-connected plane graph, and let $S$ be a minimal vertex cut of $G$. We know that $S$ contains at least three vertices and $S$ separates $G$ into two parts. We use $\textrm{Int}_{G}(S)$ and $\Ext_G(S)$ to denote the interior and the exterior of $S$, respectively. We omit $G$ if it is clear from the context.

We first show that the only positively curved $4$-regular planar graph with connectivity exactly three is the rectified triangular prism (see Figure~\ref{fig:rectifiedtriprism}). 

 \begin{lemma}\label{lem:4_reg_3_conn}
 Let $G$ be a $4$-regular positively curved planar graph such that $G$ has connectivity three. Then $G$ is the rectified triangular prism.
 \end{lemma}
 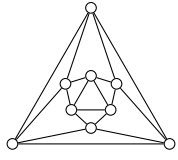
\begin{figure}[h]
\begin{center}
    \resizetikz{0.6}{\begin{tikzpicture}[scale = 1]
\foreach \k in {3,4,5}{
\pgfmathtruncatemacro{\i}{\k - 3}
\node (\k) at (90 + 120*\i:0.5) {};
}
\foreach \k in {6,7,8}{
\pgfmathtruncatemacro{\i}{\k - 6}
\node (\k) at (30 + 120*\i:0.65) {};
}
\foreach \k in {9,10,11}{
\pgfmathtruncatemacro{\i}{\k - 9}
\node (\k) at (90 + 120*\i:2) {};
}
\draw (5) -- (3) -- (4) -- (5);
\draw (5) -- (6) -- (3) -- (7) -- (4) -- (8) -- (5);
\draw (9) -- (10) -- (11) -- (9);
\draw (9) -- (7) -- (10) -- (8) -- (11) -- (6) -- (9);
\end{tikzpicture}}
\end{center}
 \caption{The rectified triangular prism}
 \label{fig:rectifiedtriprism}
 \end{figure}

\begin{proof}
    Let $G$ be a $4$-regular positively curved plane graph such that its connectivity is three. Let $S=\{u_1, u_2, u_3\}$ be a $3$-cut of $G$. Since $G$ is planar and $3$-connected, we know that $S$ separates $G-S$ into two components say $C_1$ and $C_2$. Suppose every vertex of $S$ has exactly two neighbors in each of $C_1$ and $C_2$. We show that $G$ is the rectified triangular prism, \textit{i.e.}, each $C_i$ is a triangle. Since $G$ is planar and $3$-connected, we have the following observation.
    \begin{observation}\label{obs:4-reg}
    No vertex in a component of $G-S$ is adjacent to all three vertices of $S$, and no pair of vertices of $S$ has two common neighbors in a component of $G-S$.
    \end{observation}
    We may assume that $N_{C_1}(u_1)=\{x_1, x_2\}$ and $N_{C_2}(u_1)=\{v_1, v_2\}$. First we show that $x_1x_2\in E(G)$. Suppose $x_1$ is not adjacent to $x_2$. It follows that $u_1$ and $x_i$ for each $i\in [2]$ have no common neighbor. We consider the edge $u_1x_1$. Note that $N(u_1)\setminus N[x_1]=\{v_1, v_2, x_2\}$. Since $u_1x_1$ has positive curvature, there exists a bijection $\phi: N(u_1)\setminus N[x_1] \rightarrow N(x_1)\setminus N[u_1]$ such that $d(v_1, \phi(v_1))+d(v_2, \phi(v_2))+d(x_2, \phi(x_2))\leq 4$. This implies that $d(v, \phi(v))\leq 2$ for each $v\in \{v_1, v_2, x_2\}$,  and there is at most one such $v$ with $d(v, \phi(v))=2$. We claim that $x_1$ has a neighbor in $\{u_2, u_3\}$, otherwise $d(\{v_1, v_2\}, N(x_1)\setminus N[u_1])\geq 2$ gives a contradiction. By symmetry, $x_2$ also has a neighbor in $\{u_2, u_3\}$. By Observation~\ref{obs:4-reg}, we may assume that $x_1u_2, x_2u_3\in E(G)$, and so $x_1u_3, x_2u_2\notin E(G)$.
    We now show $v_1$ has a neighbor in $\{u_2, u_3\}$. Suppose not. Since $d(v_1, u_2)\geq 2$ and $d(v_1, N(x_1)\setminus \{u_1, u_2\})\geq 3$, we have $\phi(v_1)=u_2$ and $d(v_1, \phi(v_1))=2$. Then the distance between the vertex $v_2$ and $N(x_1)\setminus \{u_1, u_2\}$ is at least two, giving a contradiction. Hence, we have $N(v_1)\cap \{u_2, u_3\}\neq \emptyset$ and $N(v_2)\cap \{u_2, u_3\}\neq \emptyset$. We may assume that $v_1u_2, v_2u_3\in E(G)$ and $v_1u_3, v_2u_2\notin E(G)$. We claim that $x_1$ and $u_3$ have a common neighbor. Since $d(v_2, N(x_1)\setminus \{u_1\})\geq 2$, we have $d(v_2, \phi(v_2))=2$. This implies that $d(v_1, \phi(v_1))=d(x_2, \phi(x_2))=1$. It follows that $\phi(v_1)=u_2$, $\phi(v_2)\in V(C_1)$ is a common neighbor of $x_1$ and $u_3$, and $\phi(x_2)\in V(C_1)$ is a common neighbor of $x_1$ and $x_2$. By symmetry, we have that $x_2$ and $u_2$ have a common neighbor in $C_1$, and observe that $\phi(v_2)$ should be this common neighbor of $x_2$ and $u_2$ by the planarity of $G$. Note that $N(\phi(v_2))=\{x_1, x_2, u_2, u_3\}$. It follows that $\{x_1, x_2\}$ is a $2$-cut of $G$, a contradiction. Therefore, $x_1x_2\in E(G)$, and so every two neighbors of $u\in S$ in $C_i$ are adjacent by symmetry.

    Next we show that $x_1$ has a neighbor in $\{u_2, u_3\}$. Suppose $N(x_1)\cap \{u_2, u_3\}=\emptyset$. We claim that $x_2$ has a neighbor in $\{u_2, u_3\}$.  If $x_2u_2, x_2u_3\notin E(G)$, we have that $d(N(u_1)\setminus N[v_1], N(v_1)\setminus N[u_1])=d(\{x_1, x_2\},N(v_1)\setminus N[u_1])\geq 2$. Moreover, we know that $v_1$ has a neighbor $w\in N(v_1)\setminus N[u_1]$ that is not contained in $S$, and so $d(\{x_1, x_2\}, w)\geq 3$. This implies that $d(x_1, \phi(x_1))+d(x_2, \phi(x_2))\geq 3+2=5$ for any bijection $\phi: \{x_1, x_2\} \rightarrow N(v_1)\setminus N[u_1]$, a contradiction. Hence, $N(x_2)\cap \{u_2, u_3\}\neq \emptyset$, and we may assume that $x_2u_3\in E(G)$ and $x_2u_2\notin E(G)$ by Observation~\ref{obs:4-reg}. We consider the edge $u_1x_1$. We know that $N(u_1)\setminus N[x_1]=\{v_1, v_2\}$ and let $N(x_1)\setminus N[u_1]=N(x_1)\setminus \{u_1, x_2\}=\{y_1, y_2\}$. Since $N(x_1)\cap \{u_2, u_3\}=\emptyset$, all distances between $\{v_1, v_2\}$ and $\{y_1, y_2\}$ are at least two. Note there exists a bijection $\phi: \{v_1, v_2\}\rightarrow\{y_1, y_2\}$ such that $d(v_1, \phi(v_1))+d(v_2, \phi(v_2))\leq 4$ as $G$ is positively curved. Thus, $d(v_1, \phi(v_1))=d(v_2, \phi(v_2))=2$. By Observation~\ref{obs:4-reg}, we may assume $v_1u_2y_1, v_2u_3y_2$ are distance two paths. Since $N_{C_1}(u_3)=\{x_2, y_2\}$, we have $x_2y_2\in E(G)$. Observe that $y_1y_2\notin E(G)$, otherwise $\{u_2, y_1\}$ is a $2$-cut.  Consider the edge $y_1u_2$. Let $N(y_1)\setminus N[u_2]=\{x_1, z\}$. Note that $z\notin N(\{u_1,u_3\})$. It follows that $d(z, N(u_2)\setminus N[y_1])\geq 3$. Since $d(x_1, N(u_2)\setminus N[y_1])\geq 2$, we have that the edge $y_1u_2$ has non-positive curvature, a contradiction. Therefore, $x_1$ has a neighbor in $\{u_2, u_3\}$, and by symmetry, we have $x_2$ has a neighbor in $\{u_2, u_3\}$. We may assume that $x_1u_2, x_2u_3\in E(G)$. This implies that $C_1$ is a triangle and similarly, $C_2$ is a triangle. Hence, $G$ is the rectified triangular prism.

   We choose a $3$-cut $S=\{u_1, u_2, u_3\}$ of $G$ such that $C$ is a minimum-vertex component of $G-S$. We know that every vertex of $S$ has at least two neighbors in $C$ by the minimality of $C$, and moreover, no vertex in $C$ is adjacent to all vertices of $S$, and no pair of vertices of $S$ has two common neighbors in $C$. Note that since $G$ is $3$-connected, every vertex in $S$ has at least one neighbor in $G-S-C$. We show that each vertex of $S$ has exactly two neighbors in $C$ and exactly two neighbors in $G-S-C$. It suffices to show each $u\in S$ has exactly two neighbors in $G-S-C$. Suppose not. We may assume that $u_1\in S$ has exactly one neighbor, say $v_1$, in $G-S-C$. Suppose $u_1u_2\notin E(G)$ and $u_1u_3\notin E(G)$. It follows that $u_1$ has three neighbors, say $x_1, x_2, x_3$, in $C$, and so, $u_1$ and $v_1$ have no common neighbor. Since $|N_C(u_1, u_2)|\leq 1$ and  $|N_C(u_1, u_3)|\leq 1$, we may assume that $x_1\notin N(u_2)\cup N(u_3)$. Let $y_1$ denote a vertex in $N(v_1)\setminus S$. We have that $d(x_1, y_1)\geq 3$. Consider a bijection $\phi:N(u_1)\setminus \{v_1\} \rightarrow N(v_1)\setminus \{u_1\}$ with $\sum_{v\in N(u_1)\setminus \{v_1\}} d(v, \phi(v))=d(x_1, \phi(x_1))+d(x_2, \phi(x_2))+d(x_3, \phi(x_3))\leq 4$. This implies that $\phi(x_1)\neq y_1$. Then since $d(x_1, N(v_1)\setminus \{u_1\})\geq 2$ and $d(N(u_1)\setminus \{v_1\},y_1)\geq 2$, such bijection $\phi$ with $\sum_{v\in N(u_1)\setminus \{v_1\}} d(v, \phi(v))\leq 4$ do not exist, giving a contradiction. We may now assume that $u_1$ is adjacent to $u_2$ or $u_3$. Since $u_1$ cannot be adjacent to both $u_2$ and $u_3$, we may assume that $u_1u_2\in E(G)$ and $u_1u_3\notin E(G)$. Observe that $u_2$ has a unique neighbor in $G-S-C$ as $u_2$ is adjacent to $u_1$ and $u_2$ has at least two neighbors in $C$. It follows that $v_1$ is not adjacent to $u_2$, otherwise $\{v_1, u_3\}$ is a $2$-cut of $G$. This implies that $v_1$ and $u_1$ has no common neighbor in $G$. Let $N_C(u_1)=\{x_1, x_2\}$. Note that there exists a bijection $\phi: N(u_1)\setminus \{v_1\} \rightarrow N(v_1)\setminus \{u_1\}$ such that $\sum_{v\in N(u_1)\setminus \{v_1\}} d(v, \phi(v))=d(u_2, \phi(u_2))+d(x_1, \phi(x_1))+d(x_2, \phi(x_2))\leq 4$. Observe that at most one vertex of $x_1, x_2$ is adjacent to $u_3$. We may assume that $x_1u_3\notin E(G)$. Then, $d(x_1, N(v_1)\setminus\{u_1\})\geq 2$ and so we have that $d(x_1,\phi(x_1))=2, d(x_2, \phi(x_2))=1, d(u_2, \phi(u_2))=1$. $d(x_2, \phi(x_2))=1$ implies that $u_3=\phi(x_2)$ is adjacent to $v_1$ and $x_2$, and $d(u_2, \phi(u_2))=1$ implies that $\phi(u_2)$ is adjacent to $u_2$ and $v_1$. Note that $\phi(u_2)$ is the unique neighbor of $u_2$ in $G-S-C$. Hence, $\phi(x_1)$ is not adjacent to $u_2$. Then $d(x_1, \phi(x_1))=2$ implies that $\phi(x_1)$ and $x_1$ are adjacent to $u_3$, contradicting that $x_1u_3\notin E(G)$. Therefore, every vertex of $S$ has exactly two neighbors in each of $C$ and $G-S-C$. Hence, $G$ is the rectified triangular prism if $G$ has connectivity three.
\end{proof}
We next classify all of the positively curved $4$-regular $4$-connected planar graphs. We first show that such a graph $G$ has a separating $4$-cycle if and only $G$ is a planar triangulation (equivalently, if and only if $G$ is the octahedron). 
\begin{lemma}\label{lem:sep_4-cycle}
    For a $4$-regular $4$-connected positively curved planar graph $G$, $G$ has a separating $4$-cycle if and only if $G$ is a planar triangulation (that is, $G$ is the octahedron).
\end{lemma}
\begin{proof}
    Let $G$ be a $4$-regular $4$-connected positively curved plane graph. The backward direction is straightforward. Then it suffices to show that if $G$ has a separating $4$-cycle, say $C=u_1u_2u_3u_4u_1$, then $G$ is the octahedron.
To show $G$ is the octahedron, we need to show that both $\Int(C)$ and $\Ext(C)$ have exactly one vertex. Suppose $\Int(C)$ has more than one vertex. Observe that each $u_i$ has exactly one neighbor in $\Int(C)$ as $G$ is $4$-connected. We claim that for each $x\in N_{\Int(C)} (V(C))$, $N(x)\cap V(C)$ has only one vertex. If $xu_i, xu_j\in E(G)$ for some distinct $i,j\in [4]$, then $x$ with the other two vertices in $V(C)\setminus \{u_i, u_j\}$ forms a $3$-cut of $G$. Let $x_i$ denote the unique neighbor of $u_i$ in $\Int(C)$. It follows that $x_1, x_2, x_3, x_4$ are pairwise distinct. We consider the edge $u_1x_1$. Let $v_1$ denote the unique neighbor of $u_1$ in $\Ext(C)$. Then, $ N(u_1)\setminus N[x_1]=\{v_1, u_2, u_4\}$ and we know that there is a bijection $\phi: \{v_1, u_2, u_4\} \rightarrow N(x_1)\setminus \{u_1\}$ such that $d(v_1, \phi(v_1))+d(u_2, \phi(u_2))+d(u_4, \phi(u_4))\leq 4$. Observe that $d(v_1,N(x_1)\setminus \{u_1\})\geq 2$. It follows that $d(v_1, \phi(v_1))=2, d(u_2, \phi(u_2))=d(u_4, \phi(u_4))=1$. This implies that $\phi(u_2)=x_2, \phi(u_4)=x_4, \phi(v_1)=x_3$. Hence, $x_1x_i\in E(G)$ for each $i\in \{2, 3,4\}$. By symmetry, we have that $x_ix_j\in E(G)$ for any distinct $i,j\in [4]$, which contradicts the planarity of $G$. Therefore, $\Int(C)$ contains only one vertex, and by symmetry, $\Ext(C)$ contains only one vertex. Since $G$ is $4$-regular, $G$ is the octahedron.
\end{proof}

The following corollary is immediate from Lemma~\ref{lem:sep_4-cycle}.

\begin{corollary}\label{cor:4-cycle}
   Let $G$ be a $4$-regular $4$-connected positively curved plane graph such that $G$ is not the octahedron. Suppose $u_1u_2$ is an edge in $G$ and $d(x_1, x_2)=1$ for some $x_1\in N(u_1)\setminus N[u_2], x_2\in N(u_2)\setminus N[u_1]$. Then $u_1u_2x_2x_1u_1$ is a facial cycle of length four.
\end{corollary}

We now use Corollary~\ref{cor:4-cycle} to classify all of the $4$-regular 4-connected positively curved planar graphs which are not planar triangulations.  

\begin{lemma}\label{lem:4-reg_no_sep}
    Let $G$ be a $4$-regular $4$-connected positively curved planar graph such that $G$ is not the octahedron. Then $G$ is  the square antiprism,  the cuboctahedron, or the pentagonal antiprism.
\end{lemma}
\begin{proof}
    Let $G$ be a $4$-regular $4$-connected positively curved planar graph such that $G$ is not the octahedron. Then $G$ is not a planar triangulation, and we choose a facial cycle $C=u_1u_2\ldots u_ku_1$ of maximum length. Since $G$ is $4$-connected, we know that for any $u,v\in V(C)$ with $uv\notin E(C)$, we have $uv\notin E(G)$ and $u,v$ cannot have a common neighbor in $G-C$.

    Suppose $k\geq 5$. We first show that the other face containing $u_1u_2$ is bounded by a triangle. Suppose the other face containing $u_1u_2$ is bounded by $C_1$ of length at least four. We consider the edge $u_1u_2$. Since $G$ is $4$-connected, $u_1$ and $u_2$ have no common neighbor. By Corollary~\ref{cor:4-cycle}, if $d(x_1, x_2)=1$ for some $x_1\in N(u_1)\setminus N[u_2], x_2\in N(u_2)\setminus N[u_1]$, then $C_1=u_1u_2x_2x_1u_1$. Hence, for any bijection $\phi: N(u_1)\setminus N[u_2]\rightarrow N(u_2)\setminus N[u_1]$, we have $\sum_{v\in N(u_1)\setminus N[u_2]} d(v, \phi(v))\geq 1+2+2=5$, a contradiction. Hence, the other face containing $u_1u_2$ is bounded by a triangle. By symmetry, this holds for all $u_iu_{i+1}$ for $i\in [k]$, and let $x_i$ denote the unique common neighbor of $u_i, u_{i+1}$. We know that $x_1, x_2, \ldots, x_k$ are pairwise distinct as $G$ is $4$-connected, and $N(V(C))=\{x_1, x_2, \ldots, x_k\}$. Now we show that for each $j\in [3, k-1]$, $x_1x_j\notin E(G)$ and $x_1,x_j$ cannot have a common neighbor in $G-C-N(V(C))$. Suppose $x_1x_j\in E(G)$, and let $D=x_1u_2 \ldots u_j x_jx_1$. We may assume that $x_2$ is contained in $\Int(D)$ and $x_k$ is contained in $\Ext(D)$. Let $y_1$ denote the neighbor of $x_1$ in $N(x_1)\setminus \{u_1, u_{2}, x_{j}\}$. If $y_1$ is contained in $\Int(D)$, then $\{u_1, u_j, x_j\}$ is a $3$-cut of $G$, a contradiction. Otherwise, if $y_1$ is contained in $\Ext(D)$, then $\{u_2, u_{j+1}, x_j\}$ is a $3$-cut of $G$, a contradiction. Hence, $x_1x_j\notin E(G)$. Suppose that $y$ is a common neighbor of $x_1,x_j$ in $G-C-N(V(C))$. Let $D=x_1u_2 \ldots u_j x_jyx_1$, and we may assume that $x_2$ is contained in $\Int(D)$ and $x_k$ is contained in $\Ext(D)$. Let $y_1$ be the neighbor of $x_1$ such that $y_1\neq u_1, u_2, y$, and let $y_j$ be the neighbor of $x_j$ such that $y_j\neq u_j, u_{j+1}, y$. If both $y_1, y_j$ are contained in $\Int(D)$, then $\{u_1,y, u_{j+1}\}$ is a $3$-cut of $G$. By symmetry, we may assume that $y_1\in V(\Int(D))$ and $y_j\in V(\Ext(D))$. Then $\{u_1, y, u_j\}$ is a $3$-cut. Hence, $x_1,x_j$ have no common neighbor in $G-C-N(V(C))$. Therefore, it follows from symmetry that for each $1\le i < j \leq k$ with $(i,j)\neq (1,k)$ and $j-i>1$, we have that $x_ix_j\notin E(G)$ and $x_i,x_j$ have no common neighbor in $G-C-N(V(C))$.

    Note that $N(u_1)\setminus N[u_2]=\{x_k, u_k\}$ and $N(u_2)\setminus N[u_1]=\{x_2, u_3\}$. Since $G$ is positively curved, there exists a bijection $\phi: \{x_k, u_k\} \rightarrow \{x_2, u_3\}$ such that $d(x_k, \phi(x_k))+d(u_k,\phi(u_k))\leq 4$. Observe that $d(\{x_k, u_k\}, \{x_2, u_3\})\geq 2$. Note that $x_k$ is not adjacent to $x_2, x_3$ since $k\geq 5$, $x_k$ is not adjacent to $u_2, u_4$ as $G$ is $4$-connected. This implies that $d(x_k, u_3)\geq 3$, and similarly $d(u_k, x_2)\geq 3$. It follows that $\phi(x_k)=x_2$ and $\phi(u_k)=u_3$. Then $d(x_k, x_2)=2$ and $d(u_k, u_3)=2$. $d(x_k, x_2)=2$ implies that $x_kx_1, x_1x_2\in E(G)$, and $d(u_k, u_3)=2$ implies that $k=5$. Thus, $C$ is $5$-cycle and $\{x_1, x_2, \ldots, x_5\}$ induces a $5$-cycle in $G$. Therefore, $G$ is the pentagonal antiprism if $k\geq 5$.

    We may now assume that all faces of $G$ are bounded by $4$-cycles or triangles. Let $C=u_1u_2u_3u_4u_1$ be a facial cycle of $G$. We show that the other facial cycle containing $u_1u_2$ is a triangle. Suppose not. Then the other facial cycle containing $u_1u_2$ is a $4$-cycle, say $C_1=u_1u_2x_2x_1u_1$. It follows that $u_1,u_2$ have no common neighbor as $G$ is $4$-connected. Let $y_1\in N(u_1)\setminus \{u_2, u_4, x_1\}$ and $y_2\in N(u_2)\setminus \{u_1, u_3, x_2\}$. We claim that $y_1, y_2$ have a common neighbor, say $z$, in $G-C-C_1$.
    Note that $N(u_1)\setminus N[u_2]=\{x_1, u_4, y_1\}$ and $N(u_2)\setminus N[u_1]=\{x_2, u_3, y_2\}$.  We know that there exists a bijection $\phi: \{x_1, u_4, y_1\} \rightarrow \{x_2, u_3, y_2\}$ such that $d(x_1, \phi(x_1))+d(u_4,\phi(u_4))+d(y_1,\phi(y_1))\leq 4$. By Corollary~\ref{cor:4-cycle}, we know that for $v\in N(u_1)\setminus N[u_2]=\{x_1, u_4, y_1\}, w\in N(u_2)\setminus N[u_1]=\{x_2, u_3, y_2\}$, if $d(v, w)=1$ then $(v, w)=(x_1,x_2)$ or $(u_4, u_3)$. This implies that $\phi(y_1)=y_2$ and $d(y_1, y_2)=2$. It follows that $y_1, y_2$ have a common neighbor $z\in V(G-C-C_1)$. Now we show that the facial cycle $C_2$ with $C_2\neq C$ containing the edge $u_2u_3$ is a triangle. Suppose $C_2$ is a $4$-cycle. Let $y_3\in N(u_3)\setminus (V(C)\cup V(C_2))$. Since $G$ has no separating $4$-cycle or triangle, we know that $y_3\notin \{z, y_1, y_2\}$. We consider the edge $u_2u_3$. Since $G$ is positively curved, we know the distance between $x_2\in N(u_2)\setminus (V(C)\cup V(C_2))$ and $y_3$ is two. Observe the common neighbor of $x_2$ and $y_3$ can only be contained in $\{z, y_1, y_2\}$. Note that $x_2y_1, y_3y_2\notin E(G)$ as $G$ has no separating $4$-cycle. It follows that $x_2z, y_3z\in E(G)$, and $N(z)=\{y_1, y_2, x_2, y_3\}$. Then $\{x_2, u_1, y_1\}$ is a $3$-cut of $G$, a contradiction. Then $C_2$ is a triangle and $y_2u_3\in E(G)$. By symmetry, we have $y_2x_2, y_1x_1, y_1u_4\in E(G)$. Note that $N(y_1)=\{z, x_1, u_1, u_4\}$ and $N(y_2)=\{z, x_2, u_2, u_3\}$. This implies that either $\{z, x_1, x_2\}$ or $\{z, u_4, u_3\}$ is a $3$-cut of $G$, giving a contradiction. Hence, the other facial cycle containing $u_1u_2$ is a triangle, and by symmetry, this holds for all $u_iu_{i+1}$ for $i\in [4]$, and let $x_i$ denote the unique common neighbor of $u_i, u_{i+1}$. We know that $x_1, x_2, x_3, x_4$ are pairwise distinct as $G$ is $4$-connected, and $N(V(C))=\{x_1, x_2, x_3, x_4\}$. Let $D_i$ denote the facial cycle containing $x_iu_{i+1}x_{i+1}$ for each $i\in [4]$. We know that each $D_i$ is a triangle or a $4$-cycle. We claim that either all $D_i$ are triangles or all $D_i$ are $4$-cycles. Suppose not. We may assume that $D_1=x_1u_2x_2y_1x_1$ is a $4$-cycle and $D_2=x_2u_3x_3x_2$ is a triangle. Observe that $x_1x_3, x_2x_4\notin E(G)$ as $G$ has no separating $4$-cycle. It follows that $y_1\notin V(C)\cup N(V(C))$. Consider the vertex $x_2$. Note that $N(x_2)=\{u_2, y_1, x_3, u_3\}$. Since two facial $4$-cycles cannot be adjacent, $y_1x_3\in E(G)$. This implies that $\{y_1, u_2, u_4\}$ is a $3$-cut, a contradiction. Hence, either all $D_i$ are triangles or all $D_i$ are $4$-cycles. If all $D_i$ are triangles, we know that $G$ is the square antiprism. We may now assume that all $D_i$ are $4$-cycles, and let $D_i=x_iu_{i+1}x_{i+1}y_ix_i$ for each $i\in [4]$. Note that each $y_i$ is contained in $G-C-N(V(C))$ and $y_1, y_2, y_3, y_4$ are pairwise distinct as $G$ has no separating $4$-cycles. Since $G$ cannot have two adjacent facial $4$-cycles, we have that $y_1y_2y_3y_4y_1$ is a $4$-cycle. This implies that $G$ is the cuboctahedron, which completes the proof.
\end{proof}
 Theorem~\ref{d4} follows immediately from Theorem~\ref{thm:3conn} and Lemmas~\ref{lem:4_reg_3_conn} and \ref{lem:4-reg_no_sep}.

\section{The positively curved $5$-regular planar graphs}

In this section, we prove Theorem~\ref{d5}, which states that there is only one $5$-regular positively curved planar graph, the icosahedral graph. For this section, we proceed using the Laplacian formulation of LLY curvature proven by M\"unch and Wojciechowski (cf. Definition~\ref{defn:llycurv}). 

Let $G$ be a connected graph and $S\subseteq V(G)$.
Suppose $S$ is a cut set of $G$, and let $C$ be a component of $G-S$. We define a function $f_{S,C}: V(G)\to \mathbb{Z}$ with respect to $S$ and $C$ as follows.
$$f_{S,C}(x):=\begin{cases}
               
                0 & \textrm{ if $x\in S$},\\
                -i & \textrm{ if $x\in N^i_C(S)$},\\
               i & \textrm{ if $x\in N^{i}_{G-S-C}(S)$}.
                 \end{cases}
$$
It can be easily shown that $f_{S,C}$ is a $1$-Lipschitz function for any vertex cut $S$ and any component $C$ of $G-S$.

By Theorem~\ref{thm:3conn}, every $5$-regular positively curved planar graph is $3$-connected. We first improve on this statement by showing that all $5$-regular positively curved planar graphs are $5$-connected.

\begin{theorem}\label{thm:5reg5conn}
    Every $5$-regular positively curved planar graph is $5$-connected.
\end{theorem}

\begin{proof}
Let $G$ be a $5$-regular positively curved planar graph. We know that $G$ is $3$-connected by Theorem~\ref{thm:3conn}. We first show that $G$ is $4$-connected. Suppose not. We choose $S\subseteq V(G)$ with $|S|=3$ such that $G-S$ has a vertex-minimum component, say $C_1$. Let $S =\{u_1,u_2,u_3\}$. Observe that each vertex of $S$ has at least two neighbors in $C_1$ as $G$ is $5$-regular and $C_1$ is vertex-minimum. By the vertex minimality of $C_1$ and planarity of $G$, we have the following observations.
\begin{observation}\label{obs:5reg1}

    All vertices of $S$ cannot have a common neighbor in $C_1$.
    \end{observation}
\begin{proof}
Suppose $u_1,u_2,u_3$ have a common neighbor in $C_1$, say $x$. Since $d(x)=5$, one of the vertex subsets $\{u_1,u_2,x\}, \{u_2,u_3,x\}, \{u_3,u_1,x\}$ must be a $3$-cut $S'$ such that one component of $G-S'$ has fewer vertices, contradicting the choice of $S$. 
\end{proof}
\begin{observation}\label{obs:5reg2}
Every two vertices of $S$ have at most one common neighbor in $C_1$.
\end{observation}
\begin{proof}
    Similar to the given proof of Observation~\ref{obs:5reg1}. 
\end{proof}
 We claim that any vertex of $S$ has at least two neighbors in $G-S-C_1$. Suppose not. Note that every vertex in $S$ has at least one neighbor that is not contained in $S\cup V(C_1)$ as $G$ is $3$-connected. Hence, without loss of generality, we may assume that $u_1$ has exactly one neighbor, say $x_1$, in $G-S-C_1$. We claim that $u_1$ is adjacent to $u_2$ and $u_3$, i.e., $u_1$ has exactly two neighbors in $C_1$. Let $S'$ denote the set $\{x_1, u_2,u_3\}$. Note that $S'$ is also a cut set of $G$ and $C_1'=G[C_1+u_1]$ is a component of $G-S'$. Consider the $1$-Lipschitz function $f_1:=f_{S',C_1'}$. Consider the edge $u_1x_1$. By the definition of $f_1$, it follows that $f_1(u_1)=-1, f_1(x_1)=0$,
\begin{align*}
\Delta f_1(u_1)&=\frac{ 1+|N(u_1)\cap \{u_2,u_3\}|-(4-|N(u_1)\cap \{u_2,u_3\}|-|N(u_1)\cap N_{C_1}(\{u_2,u_3\})|)}{5} \\&
=\frac{2|N(u_1)\cap \{u_2,u_3\}|+|N(u_1)\cap N_{C_1}(\{u_2,u_3\})|-3}{5},
\end{align*}
and 
\begin{align*}
    \Delta f_1(x_1)=\frac{-1+(4-|N(x_1)\cap \{u_2,u_3\}|)}{5}=\frac{3-|N(x_1)\cap \{u_2,u_3\}|}{5}.
\end{align*}
Hence, $$\kappa(u_1,x_1)\leq \Delta f_1(u_1)-\Delta f_1(x_1)=\frac{2|N(u_1)\cap \{u_2,u_3\}|+|N(u_1)\cap N_{C_1}(\{u_2,u_3\})|+|N(x_1)\cap \{u_2,u_3\}|-6}{5}.$$
Note that $|N(u_1)\cap N_{C_1}(\{u_2,u_3\})|\leq 2$ by Observations~\ref{obs:5reg1} and \ref{obs:5reg2}. Suppose $|N(u_1)\cap \{u_2,u_3\}|\leq 1$. Then $2|N(u_1)\cap \{u_2,u_3\}|+|N(u_1)\cap N_{C_1}(\{u_2,u_3\})|+|N(x_1)\cap \{u_2,u_3\}|\le 2\cdot 1 +2 +2\leq 6$, contradicting that $\kappa(u_1,x_1)>0$. Therefore, $|N(u_1)\cap \{u_2,u_3\}|=2$ and this implies that $u_1u_2, u_1u_3\in E(G)$. Moreover, since $\kappa(u_1,x_1)>0$, we have that $|N(x_1)\cap \{u_2, u_3\}|\ge 1$.

We define a new $1$-Lipschitz function $f_2:=f_{S,C_1}$.
We claim that for any $u\in \{u_2,u_3\}$, if $ux_1\in E(G)$ then $N(x_1)\cap N_{G-S-C_1}(u)=\emptyset$. Suppose not. Without loss of generality, we may assume $u_2x_1\in E(G)$ and $y_1\in N(x_1)\cap N_{G-S-C_1}(u_2)$ Since $u_2$ is adjacent to $u_1,x_1,y_1$, observe that $u_2$ has exactly two neighbors in $C_1$, and $N_{G-S-C_1}(S)=\{x_1, y_1\} \cup N_{G-S-C_1}(u_3)$.  Recall that each of $u_2,u_3$ has at most two neighbors in $G-S-C_1$.
Consider the edge $u_2y_1$. Note that $f_2(u_2)=0$, $f_2(y_1)=1$, and $\Delta f_2(u_2)=0$.
We show that $y_1u_3\in E(G)$. Assume that $y_1$ is not adjacent to $u_3$. Since $|N_{G-S-C_1}(u_3)|\leq 2$, we know that $N_{G-S-C_1}(S)=\{x_1, y_1\} \cup N_{G-S-C_1}(u_3)$ has at most $4$ vertices and $y_1$ has at least one neighbor contained in $N^2_{G-S-C_1}(S)$. It follows that $\Delta f_2(y_1)\ge 0$ and $\kappa(u_2, y_1)\le \Delta f_2(u_2)-\Delta f_2(y_1)\leq 0$, a contradiction. Hence, $y_1u_3\in E(G)$. Observe that $u_3x_1\notin E(G)$ and $N_{G-S-C_1}(u_3)\backslash \{y_1\}\ne \emptyset$. Otherwise, $N_{G-S-C_1}(S)=\{x_1,y_1\}$ is a $2$-cut, contradicting that $G$ is $3$-connected.
We now consider the edge $u_1x_1$. It follows that $ f_2(u_1)=0,  f_2(x_1)=1$, and $\Delta f_2(u_1)=\frac{-1}{5}$. Since $\kappa(u_1,x_1)>0$, we have that $\Delta f_2(x_1)<\frac{-1}{5}$. Note that $x_1u_3\notin E(G)$ and $N_{G-S-C_1}(S)=\{x_1, y_1\} \cup N_{G-S-C_1}(u_3)$ has three vertices. We know that $x_1$ has at least one neighbor in $N^2_{G-S-C_1}(S)$, and so $\Delta f_2(x_1)\ge \frac{-2+1}{5}=\frac{-1}{5}$, giving a contradiction. Therefore, for $u=u_2,u_3$, we have that $N(x_1)\cap N_{G-S-C_1}(u)=\emptyset$ if $x_1u\in E(G)$.

Recall that $|N(x_1)\cap \{u_2, u_3\}|\ge 1$. Without loss of generality, we may assume $u_2x_1\in E(G)$. Then $N(x_1)\cap N_{G-S-C_1}(u_2)=\emptyset$.
Note that $ f_2(u_1)=0,f_2(x_1)=1, \Delta f_2(u_1)=-\frac{1}{5}$, and
 \begin{align*}
     \Delta f_2 (x_1)&=\frac{-2-\mathbbm{1}_{u_3x_1\in E(G)}+(3-\mathbbm{1}_{u_3x_1\in E(G)}-|N(x_1)\cap N_{G-S-C_1}(\{u_2,u_3\})|)}{5}\\&=
     \frac{1-2\cdot \mathbbm{1}_{u_3x_1\in E(G)}-|N(x_1)\cap N_{G-S-C_1}(u_3)|}{5}
 \end{align*}
Hence,
 \[\kappa(u_1,x_1)\le \Delta f_2(u_1)-\Delta f_2(x_1)=\frac{-2+2\cdot \mathbbm{1}_{u_3x_1\in E(G)}+|N(x_1)\cap N_{G-S-C_1}(u_3)|}{5}.\]
We know that $|N(x_1)\cap N_{G-S-C_1}(u_3)|\le |N_{G-S-C_1}(u_3)|\leq 2$, and if $u_3x_1\in E(G)$ then $N(x_1)\cap N_{G-S-C_1}(u_3)=\emptyset$. This implies that $\kappa(u_1,x_1)\leq 0$, giving a contradiction. Therefore, we have that $u_i$ has at least two neighbors in $G-S-C_1$ for each $i\in [3]$, and hence at most three neighbors in $C_1$. 

We claim that each $u_i$ has exactly two neighbors in $C_1$. Suppose not. We may assume that $u_1$ has three neighbors in $C_1$. Then $u_1$ has no neighbors in $S$ and has exactly two neighbors in $G-S-C_1$. Let $N_{G-S-C_1}(u_1)=\{x_1,x_2\}$. It follows that $T=\{u_2,u_3, x_1,x_2\}$ is a cut set of $G$ and $C_T=G[C_1+u_1]$ is a component of $G-T$. Let $g:=f_{T,C_T}$. Consider the edge $u_1x_1$. Then $g(u_1)=-1, g(x_1)=0,$  $$\Delta g (u_1)=\frac{-(3-\mathbbm{1}_{N(u_1)\cap N_{C_1}(u_2)}-\mathbbm{1}_{N(u_1)\cap N_{C_1}(u_3)})+2}{5}=\frac{-1+\mathbbm{1}_{N(u_1)\cap N_{C_1}(u_2)}+\mathbbm{1}_{N(u_1)\cap N_{C_1}(u_3)}}{5},$$
$$\Delta g (x_1)=\frac{-1+(4-|N_{T}(x_1)|)}{5}=\frac{3-|N_T(x_1)|}{5}.$$
Thus, \[\kappa(u_1,x_1)\le \Delta g(u_1)-\Delta g(x_1)=\frac{-4+\mathbbm{1}_{N(u_1)\cap N_{C_1}(u_2)}+\mathbbm{1}_{N(u_1)\cap N_{C_1}(u_3)}+|N_T(x_1)|}{5}.\]
Since $\kappa(u_1,x_1)>0$, we obtain that $|N_T(x_1)|\ge 3$. This implies that $N_T(x_1)=\{x_2,u_2,u_3\}$. By symmetry, we consider the edge $u_1x_2$ and have that $N_T(x_2)=\{x_1,u_2,u_3\}$. This contradicts the planarity of $G$. It follows that every vertex in $S$ has exactly two neighbors in $C_1$. Observe that since $|S|=3$, there exists $u_i\in S$ such that $u_i$ has no neighbor in $S$. We may assume that $u_1$ has no neighbor in $S$, and let $N_{C_1}(u_1)=\{v_1, v_2\}$. Recall that $f_2=f_{S, C_1}$. Hence $f_2(u_1)=0, f_2(v_1)=f_2(v_2)=-1$, and $\Delta f_2(u_1)=\frac{1}{5}$. Observe that 
\begin{align*}
\Delta f_2(v_1)&=\frac{1+\mathbbm{1}_{N(v_1)\cap \{u_2,u_3\}\neq \emptyset}-(4-\mathbbm{1}_{N(v_1)\cap \{u_2,u_3\}\neq \emptyset}-|N(v_1)\cap N_{C_1}(S)|)}{5}\\&=\frac{-3+2\cdot \mathbbm{1}_{N(v_1)\cap \{u_2,u_3\}\neq \emptyset}+|N(v_1)\cap N_{C_1}(S)|}{5}.
\end{align*}
Thus, \[0<\kappa(v_1,u_1)\le \Delta f_2(v_1)-\Delta f_2(u_1)=\frac{-4+2\cdot \mathbbm{1}_{N(v_1)\cap \{u_2,u_3\}\neq \emptyset}+|N(v_1)\cap N_{C_1}(S)|}{5}.\]
Since $|N(v_1)\cap N_{C_1}(S)|\le 4-\mathbbm{1}_{N(v_1)\cap \{u_2,u_3\}\neq \emptyset}$, it follows that $N(v_1)\cap \{u_2,u_3\}\neq \emptyset$. By symmetry, we know that $N(v_2)\cap \{u_2,u_3\}\neq \emptyset$. We may assume $v_1$ is adjacent to $u_2$ and $v_2$ is adjacent to $u_3$. This implies that $u_2$ and $u_3$ have no common neighbor in $C_1$, otherwise $N_{C_1}(S)$ has size three and is a $3$-cut of $G$ such that $C_1-N_{C_1}(S)$ has fewer vertices than $C_1$ contradicting our choice of $S$. Thus, we may assume that $N_{C_1}(u_2)\backslash N(u_3)=\{v_3\}$ and $N_{C_1}(u_3)\backslash N(u_2)=\{v_4\}$, i.e., $N_{C_1}(S)=\{v_1,v_2,v_3,v_4\}$. Observe that $N(v_3)\cap S=\{u_2\}$, and $v_3$ has at least one neighbor in $N^2_{C_1}(S)$. Thus $\Delta f_2(v_3)\leq \frac{-1+1}{5}=0$. But $\Delta f_2(u_2)\geq \frac{-2+2}{5}=0$ implies that $\kappa(v_3,u_2)\le \Delta f_2(v_3)-\Delta f_2(u_2)\leq 0$, giving a contradiction. Therefore, we obtain that $G$ is $4$-connected.

Suppose $G$ is not $5$-connected. Let $S=\{u_1,u_2,u_3,u_4\}$ be a $4$-cut of $G$. By planarity and $3$-connectedness of $G$, we know that the four vertices in $S$ admit an order, say $u_1u_2u_3u_4$, such that $u_i$ can be adjacent to $u_{i-1}, u_{i+1}$ but cannot be adjacent to $u_{i+2}$ for each $i\in [4]$, where the indices are the same modulo $4$. We may choose $S=\{u_1,u_2,u_3,u_4\}$ (with the order $u_1u_2u_3u_4$) such that $G-S$ has a minimum-vertex component $C_1$. By the choice of $S$ and the $4$-connectedness of $G$, we know that $G[C_1+S]$ has no cut of size at most four. Hence, we have the following observations.
\begin{observation}\label{obs:5reg3}
For each $i\in [4]$,
\begin{itemize}
\item $u_i$ has at least two neighbors in $C_1$,
    \item $u_i$ and $u_{i+1}$ have at most one common neighbor in $C_1$, and
    \item $u_i$ and $u_{i+2}$ have no common neighbor in $C_1$.
\end{itemize}    
\end{observation}
Similarly, we show that every vertex of $S$ has at least two neighbors in $G-S-C_1$. Suppose not. We may assume that $u_1$ has exactly one neighbor $x_1$ in $G-S-C_1$. Observe that $S'=\{x_1, u_2, u_3, u_4\}$  is also a $4$-cut of $G$ and $C_1'=G[C_1+u_1]$ is one component of $G-S'$. Let $g_1:=f_{S',C_1'}$. Note that $g_1(u_1)=-1$ and $g_1(x_1)=0$. Consider the vertex $u_1$ and its neighbors. Note that $u_1u_3\notin E(G), N_{C_1}(u_1)\cap N_{C_1}(u_3)=\emptyset$, and $|N_{C_1}(u_1)\cap N_{C_1} (u_j)|\le 1$ for $j=2,4$. Then it follows that
\begin{align*}
\Delta g_1(u_1)&=\frac{ 1+|N(u_1)\cap \{u_2,u_4\}|-(4-|N(u_1)\cap \{u_2,u_4\}|-|N(u_1)\cap N_{C_1}(\{u_2,u_4\})|)}{5} \\&
=\frac{2|N(u_1)\cap \{u_2,u_4\}|+|N(u_1)\cap N_{C_1}(\{u_2,u_4\})|-3}{5}.
\end{align*}
Consider the vertex $x_1$ and its neighbors. Note that $S'=\{x_1, u_2,u_3, u_4\}$ has the order $x_1u_2u_3u_4$. Thus, $x_1u_3\notin E(G)$ and we have that
\begin{align*}
    \Delta g_1(x_1)=\frac{-1+(4-|N(x_1)\cap \{u_2,u_4\}|)}{5}=\frac{3-|N(x_1)\cap \{u_2,u_4\}|}{5}.
\end{align*}
Therefore, it follows that
$$\kappa (u_1,x_1)\le \Delta g_1(u_1)-\Delta g_1(x_1)=\frac{2|N(u_1)\cap \{u_2,u_4\}|+|N(u_1)\cap N_{C_1}(\{u_2,u_4\})|+|N(x_1)\cap \{u_2,u_4\}|-6}{5}. $$
Similarly, since $G$ is positively curved, we obtain that 
$u_1$ is adjacent to $u_2$ and $u_4$, and that $N(x_1)\cap \{u_2,u_4\}\neq \emptyset$. Let $S''=\{u_1,u_2,u_4\}\cup N_{C_1}(u_3)$. Since $u_1$ has at least two neighbors in $C_1$ and $N(u_1)\cap N_{C_1}(u_3)=\emptyset$, it follows that $C_1''=C_1-N_{C_1}(u_3)$ is not empty and it is a component of $G-S''$. We let $g_2:=f_{S'', C_1''}$. Observe that $g_2^{-1}(1)=N_{G-S''-C_1''}(S'')=\{u_3, x_1\}\cup N_{G-S-C_1}(\{u_2,u_4\})$. Consider the edge $u_1x_1$. Note that $g_2(u_1)=0$ and $g_2(x_1)=1$. Since $N(u_1)\cap S''=\{u_2,u_4\}$, we have that $\Delta g_2(u_1)=-\frac{1}{5}$. We next claim that $x_1$ is adjacent to $u_2$ and $u_4$. Suppose not. Since we know that $N(x_1)\cap \{u_2,u_4\}\neq \emptyset$, we may assume that $x_1u_2\in E(G)$ and $x_1u_4\notin E(G)$. Note that $x_1u_3\notin E(G)$. This implies that
\begin{align*}
\Delta g_2(x_1)&=\frac{-1-\mathbbm{1}_{x_1u_2\in E(G)}-\mathbbm{1}_{x_1u_4\in E(G)}+(4-\mathbbm{1}_{x_1u_2\in E(G)}-\mathbbm{1}_{x_1u_4\in E(G)}-|N(x_1)\cap N_{G-S-C_1} (\{u_2,u_4\})|)}{5}
\\&=\frac{3-2\cdot\mathbbm{1}_{x_1u_2\in E(G)}-2\cdot \mathbbm{1}_{x_1u_4\in E(G)}-|N(x_1)\cap N_{G-S-C_1} (\{u_2,u_4\})| }{5}
\\&=\frac{1-|N(x_1)\cap N_{G-S-C_1} (\{u_2,u_4\})|}{5},
\end{align*}
and \[\kappa(u_1,x_1)\le \Delta g_2(u_1)-\Delta g_2(x_1)=\frac{-2+|N(x_1)\cap N_{G-S-C_1} (\{u_2,u_4\})|}{5}.\]
Thus, we have that $|N(x_1)\cap N_{G-S-C_1} (\{u_2,u_4\})|\geq 3$. Observe that for each $j=2,4$, we know that $N(x_1)\cap N_{G-S-C_1} (u_j)$ has at most one vertex, otherwise we obtain a $3$-cut of $G$. It follows that $|N(x_1)\cap N_{G-S-C_1} (\{u_2,u_4\})|\leq 2$, a contradiction. Hence, $x_1u_2,x_1u_4\in E(G)$, and so $$\Delta g_2(x_1)=\frac{-1-|N(x_1)\cap N_{G-S-C_1} (\{u_2,u_4\})|}{5}.$$
Then $$0<\kappa(u_1,x_1)\le \Delta g_2(u_1)-\Delta g_2(x_1)= \frac{|N(x_1)\cap N_{G-S-C_1} (\{u_2,u_4\})|}{5}$$ implies that $N(x_1)\cap N_{G-S-C_1} (\{u_2,u_4\})\neq \emptyset$. Let $y_1\in N(x_1)\cap N_{G-S-C_1} (\{u_2,u_4\})$. Without loss of generality, we may assume that $y_1\in  N_{G-S-C_1} (u_2)$. Since each of $u_2,u_4$ has at most two neighbors in $G-S-C_1$ and $G$ is $4$-connected, we know that $y_1u_4\notin E(G)$ and $g_2^{-1}(1)=\{u_3, x_1, y_1\}\cup N_{G-S-C_1}(u_4)$ has exactly four vertices. We may assume that $N_{G-S-C_1}(u_4)\backslash\{x_1\}=\{z_1\}$, i.e., $g_2^{-1}(1)=\{u_3, x_1, y_1, z_1\}$. We consider the edge $u_2y_1$. Observe that 
$g_2(u_2)=0, g_2(y_1)=1$, 
$$\Delta g_2(u_2)=\frac{-(2-\mathbbm{1}_{N_{C_1} (u_2) \cap N_{C_1}(u_3)\neq \emptyset})+2}{5}=\frac{\mathbbm{1}_{N_{C_1} (u_2) \cap N_{C_1}(u_3)\neq \emptyset}}{5},$$
and $$\Delta g_2(y_1)=\frac{-1+(4-|N(y_1)\cap \{u_3,x_1,z_1 \} |)}{5}=\frac{2-|N(y_1)\cap \{u_3,z_1 \}|}{5}.$$
Thus, $$0<\kappa (u_2, y_1)\leq \Delta g_2(u_2)-\Delta g_2(y_1) =\frac{\mathbbm{1}_{N_{C_1} (u_2) \cap N_{C_1}(u_3)\neq \emptyset}+|N(y_1)\cap \{u_3,z_1 \} |-2}{5}$$ 
implies that $|N(y_1)\cap \{u_3,z_1 \} |\geq 2$. It follows that $\{x_1, u_3, z_1\}\subseteq N(y_1)$. This implies that either $\{x_1, y_1, z_1\}$ or $\{u_3, y_1, z_1\}$ is a $3$-cut of $G$, a contradiction. Therefore, each vertex of $S$ has at least two neighbors in $G-S-C_1$, and so has at most three neighbors in $C_1$.

We claim that each vertex of $S$ has exactly two neighbors in $C_1$. Suppose not. We may assume $u_1$ has three neighbors in $C_1$, and hence has exactly two neighbors in $G-S-C_1$, say $x_1, x_2$. Recall that $S''=\{u_1,u_2,u_4\}\cup N_{C_1}(u_3)$, $C_1''=C_1-N_{C_1}(u_3)$, and $g_2=f_{S'', C_1''}$.  Consider the edges $u_1x_1, u_1x_2$. We know that $g_2(u_1)=0, g_2(x_1)=g_2(x_2)=1$, and 
$\Delta g_2(u_1)=\frac{-1}{5}$. Since $G$ is $4$-connected, $N(x_i)\cap \{u_2,u_4\}$ has at most one vertex for each $i\in [2]$, and $x_1,x_2$ cannot have a common neighbor in $\{u_2, u_4\}$. Since $G$ is positively curved, we obtain that $N(x_i)\cap \{u_2,u_4\}\neq \emptyset$, and $N(x_i)\backslash \{u_1, u_2, u_4\}\subseteq g_2^{-1}(1)$. We may assume that $x_1u_2\in E(G)$ and $x_2u_4\in E(G)$. Note that $g_2^{-1}(1)=\{u_3, x_1, x_2\}\cup N_{G-S-C_1}(\{u_2, u_4\})$. Observe that $x_iu_3\notin  E(G)$ for each $i\in [2]$, otherwise $u_2$ has a unique neighbor, $x_1$, in $G-S-C_1$ or $G$ has a $3$-cut. Since $x_1$ has at most one neighbor in $N_{G-S-C_1}(u_2)$ and $N(x_1)\backslash \{u_1, u_2\}\subseteq g_2^{-1}(1)=\{u_3, x_1, x_2\}\cup N_{G-S-C_1}(\{u_2, u_4\})$, we know that $x_1$ has at least one neighbor, say $y_1$, in $N_{G-S-C_1}(u_4)\backslash \{x_2\}$. Similarly, we obtain that $x_2$ has at least one neighbor in $N_{G-S-C_1}(u_2)\backslash \{x_1\}$. By planarity, it follows that $y_1$ is the common neighbor of $x_2$ and $u_2$ in $G-S-C_1-x_1$. But this implies that $N(x_1)\cap g_2^{-1}(1)\subseteq \{x_2, y_1\}$, and so $\Delta g_2(x_1)\ge \frac{-1}{5}$, contradicting that $\kappa (u_1, x_1)>0$. Thus, $u_i$ has exactly two neighbors in $C_1$ for each $i\in [4]$.

Next we show that for each vertex $v\in N_{C_1}(S)$, $v$ has exactly two neighbors in $S$. Note that $1\leq |N_S(v)|<3$ for each $v\in N_{C_1}(S)$. Suppose $N_{C_1}(u_1)=\{v_1, v_2\}$ and assume that $N_S(v_1)$ contains only one vertex $u_1$. Consider the $1$-Lipschitz function $g=f_{S,C_1}$ and the edge $v_1u_1$. Note that $g(v_1)=-1, g(u_1)=0$ and $u_1$ has at least two neighbors in $G-S-C_1$. Hence $\Delta g(u_1)\ge 0$. Since $N(v_1)\cap S=\{u_1\}$, we have that 
\[\Delta g(v_1)=\frac{1-(4-|N(v_1)\cap N_{C_1}(S)|)}{5}=\frac{-3+|N(v_1)\cap N_{C_1}(S)|}{5}.\]
Thus, \[0 < \kappa(v_1, u_1)\le \Delta g(v_1)-\Delta g(u_1)\le \frac{-3+|N(v_1)\cap N_{C_1}(S)|}{5}\]
implies that $|N(v_1)\cap N_{C_1}(S)|=4$. By our choice of $S$, we have that $|N(v_1)\cap N_{C_1}(u_3)|\leq 2$ and $|N(v_1)\cap N_{C_1}(u_j)|\leq 1$ for $j\in \{2,4\}$.  Note that either $v_2u_2\notin E(G)$ or $v_2u_4\notin E(G)$. Without loss of generality, we may assume that $v_2u_2\notin E(G)$. Suppose $v_2u_4\in E(G)$. Then we know that $|N(v_1)\cap N_{C_1}(u_3)|= 2$, $|N(v_1)\cap N_{C_1}(u_2)|=1$ and $N(v_1)\cap N_{C_1}(u_4)=\{v_2\}$. This implies that $u_2, u_3, v_1$ and one vertex in $N(v_1)\cap N_{C_1}(u_3)$ is a $4$-cut of $G$, a contradiction. We may now assume that $v_2u_4\notin E(G)$. Since $G$ is $4$-connected, we have that $|N(v_1)\cap N_{C_1}(u_2)|+|N(v_1)\cap N_{C_1}(u_4)|\leq 1$. Hence we obtain that $v_1v_2\in E(G), |N(v_1)\cap N_{C_1}(u_2)|+|N(v_1)\cap N_{C_1}(u_4)|= 1$, and $|N(v_1)\cap N_{C_1}(u_3)|= 2$. We may assume that $|N(v_1)\cap N_{C_1}(u_2)|=1$. Similarly, we obtain a $4$-cut of $G$ formed by  $u_2, u_3, v_1$ and one vertex in $N(v_1)\cap N_{C_1}(u_3)$, giving a contradiction. Hence, each vertex in $ N_{C_1}(S)$ has exactly two neighbors in $S$, and this implies that $N_{C_1}(S)$ contains exactly four vertices, contradicting our choice of $S$. Therefore, $G$ is $5$-connected and the proof is complete.
\end{proof}
\begin{lemma}\label{lem:5regplanartriangulation}
    Every $5$-connected positively curved $5$-regular planar graph is a planar triangulation.
\end{lemma}
\begin{proof}
Let $G$ be a positively curved $5$-regular planar graph such that $G$ is $5$-connected. Suppose $G$ is not a planar triangulation. Then $G$ has a face, whose boundary $C$ is a cycle of length at least four. We first show that $C$ is a $4$-cycle. Suppose $|V(C)|\ge 5$. Let $u_1u_2$ be an edge in $C$. Let $D$ denote the other facial cycle containing $u_1u_2$ in $G$, and let $x_i\neq u_{3-i}$ denote the other neighbor of $u_i$ in $D$. Since $G$ is $5$-connected, we know that $u_1, u_2$ can have at most one common neighbor, and if $N(u_1)\cap N(u_2)\neq \emptyset$ then $x_1=x_2$ and $D=u_1u_2x_1u_1$ is a triangle. Hence, $N(u_1)\backslash\{u_2,x_1\}$ and $N[u_2]\backslash\{u_1, x_1\}$ are disjoint and there is no edge between these two sets. We consider the edge $u_1u_2$ and define a $1$-Lipschitz function $f: N[u_1]\cup N[u_2]\to \mathbb{Z}$ such that
$$f(x):=\begin{cases}
               
                -1 & \textrm{ if $x\in N(u_1)\backslash\{u_2,x_1\}$},\\
                0 & \textrm{ if $x\in \{u_1,x_1\}$},\\
               1 & \textrm{ if $x\in N[u_2]\backslash\{u_1, x_1\}$}.
                 \end{cases}
$$
It follows that $\Delta f(u_1)=\frac{-3+1}{5}=\frac{-2}{5}$ and $\Delta f(u_2)\geq \frac{-2}{5}$. Hence, $\kappa(u_1, u_2)\leq \Delta f(u_1)-\Delta f(u_2)\le 0$, giving a contradiction. Thus, we may assume $C$ has length four. 

Let $C=u_1u_2u_3u_4u_1$. For convenience, all the indices are the same modulo $4$. Since $G$ is $5$-connected, we know $|N(u_i)\cap N(u_{i+1})|\leq 1$ and $N(u_i)\cap N(u_{i+2})=\{u_{i-1}, u_{i+1}\}.$ We claim that $u_i$ and $u_{i+1}$ have exactly one common neighbor for each $i\in [4]$. Assume that $N(u_1)\cap N(u_2)=\emptyset$. Let $D$ ($D\neq C$) be the other facial cycle containing the edge $u_1u_2$, and let $N_{D}(u_i)=\{u_{3-i}, x_i\}$ for each $i\in[2]$. Consider the following $1$-Lipschitz function $f_1: N[u_1]\cup N[u_2]\to \mathbb{Z}$ such that
$$f_1(x):=\begin{cases}
               
                -1 & \textrm{ if $x\in N(u_1)\backslash\{u_2,x_1, u_4\}$},\\
                0 & \textrm{ if $x\in \{u_1,x_1, u_4\}$},\\
               1 & \textrm{ if $x\in N[u_2]\backslash\{u_1\}$}.
                 \end{cases}$$
Hence, we have $f_1(u_1)=0, f_1(u_2)=1, \Delta f_1(u_1)=\frac{-1}{5}$ and $\Delta f_1(u_2)=\frac{-1}{5}$, contradicting that $\kappa(u_1, u_2)>0$. Therefore, $u_i$ and $u_{i+1}$ have exactly one common neighbor, say $x_i$, in $G$ for each $i\in [4]$. Let $y_i$ denote the unique vertex in $N(u_i)\backslash (V(C)\cup \{x_i, x_{i-1}\})$. 

Next, we show that there exists some $i_0\in [4]$ such that the distance between $x_{i_0}$ and $\{x_{i_0+2}, y_{i_0+2}\}$ is three.
Suppose $d(x_1,\{y_3, x_3\})<3$. Assume that $d(x_1,y_3)<3$. Since $G$ is $5$-connected, we know $x_1y_3\notin E(G)$ and so $d(x_1,y_3)\geq 2$. Hence, we know that the distance of $x_1$ and $y_3$ is two, and $N(x_1)\cap N(y_3)\neq \emptyset$. We may assume that $z\in N(x_1)\cap N(y_3)$. Observe that $z\in N^2(V(C))$. Then we obtain that $x_4$ is not adjacent to $z$, otherwise $\{x_4, u_1, x_1, z\}$ is a $4$-cut of $G$. This implies that the distance between $x_4$ and $\{x_2, y_2\}$ is three. Similarly, $d(x_4,\{y_2, x_2\})=3$ if $d(x_1, x_3)<3$. Thus, we obtain that $d(x_{i_0}, \{x_{i_0+2}, y_{i_0+2}\})=3$ for some $i_0\in [4]$. Without loss of generality, we may assume that $d(x_4,\{y_2, x_2\})=3$. Now we consider the edge $u_1u_2$ and the following $1$-Lipschitz function $g: N[u_1]\cup N[u_2]\to \mathbb{Z}$ with
$$g(x):=\begin{cases}
               
                -1 & \textrm{ if $x=x_4$},\\
                0 & \textrm{ if $x\in \{u_1, y_1, u_4\}$},\\
               1 & \textrm{ if $x\in \{u_2, x_1, u_3\}$},\\
               2 & \textrm{ if $x\in \{x_2,y_2\}$}.
                 \end{cases}$$
Note that $g(u_1)=0, g(u_2)=1, \Delta g(u_1)=\frac{1}{5}$, and $\Delta g (u_2)=\frac{1}{5}$. It follows that $\kappa(u_1, u_2)\leq \Delta g(u_1)-\Delta g(u_2)=0$, giving a contradiction. Thus, $G$ is a planar triangulation.
\end{proof}

As the icosahedral graph is the only $5$-regular planar graph which is a planar triangulation, Theorem~\ref{thm:5reg5conn} and Lemma~\ref{lem:5regplanartriangulation} together prove Theorem~\ref{d5}. 

\section{Future work}
In this paper, we have proven that every positively curved $d$-regular graph has diameter at most $2d-2$, improving on the upper bound $2d$ obtained by Lin, Lu and Yau. We believe that a stronger bound should hold.

\begin{conjecture}\label{con:dregdiamconj}
Let $G$ be a positively curved $d$-regular graph. Then
\[\diam(G) \le d + C,\]
where $C$ is some absolute constant independent of $d$.
\end{conjecture}
In fact, we do not know any example of a positively curved $d$-regular graph $G$ with $\diam(G) > d$, so it is possible that the statement of Conjecture~\ref{con:dregdiamconj} is true with $C=0$. The hypercubes $Q_d$ show that a positively curved $d$-regular graph can have diameter $d$.

\section*{Acknowledgment}

We thank Teegan Bailey for helpful discussions in an early stage of the paper. During the preparation of this manuscript, the authors used ChatGPT solely to help identify typographical and grammatical errors. AI-assisted tools were not used to generate the mathematical ideas, results, proofs, or arguments presented in this paper.

\printbibliography
\end{document}